\documentclass[a4paper]{amsart}
\pdfoutput=1
\usepackage{etex}

\usepackage[T1]{fontenc}
\usepackage{lmodern}
\usepackage[utf8]{inputenc}
\usepackage{amssymb}
\usepackage{enumitem}
\setlist[enumerate,1]{label=\textup{(\arabic*)}}

\usepackage{mathrsfs,dsfont,mathtools}
\usepackage{nicefrac}

\usepackage{xcolor}

\usepackage{tikz}
\usetikzlibrary{matrix,arrows,calc}
\tikzset{cd/.style=matrix of math nodes,row sep=2em,column sep=2em, text height=1.5ex, text depth=0.5ex}
\tikzset{cdar/.style=->,auto}
\tikzset{mid/.style={anchor=mid}} 
\tikzset{dar/.style={double,double equal sign distance,-implies}}
\tikzset{narrowfill/.style={inner sep=1pt, fill=white}}
\usepackage{tikz-cd}

\usepackage[pdftitle={Correspondences between graph C*-algebras},
  pdfauthor={Rasmus Bentmann, Ralf Meyer},
  pdfsubject={Mathematics}
]{hyperref}
\usepackage[lite]{amsrefs}

\renewcommand*{\PrintDOI}[1]{\href{http://dx.doi.org/\detokenize{#1}}{doi: \detokenize{#1}}}

\BibSpec{book}{%
  +{}  {\PrintPrimary}                {transition}
  +{,} { \textit}                     {title}
  +{.} { }                            {part}
  +{:} { \textit}                     {subtitle}
  +{,} { \PrintEdition}               {edition}
  +{}  { \PrintEditorsB}              {editor}
  +{,} { \PrintTranslatorsC}          {translator}
  +{,} { \PrintContributions}         {contribution}
  +{,} { }                            {series}
  +{,} { \voltext}                    {volume}
  +{,} { }                            {publisher}
  +{,} { }                            {organization}
  +{,} { }                            {address}
  +{,} { \PrintDateB}                 {date}
  +{,} { }                            {status}
  +{}  { \parenthesize}               {language}
  +{}  { \PrintTranslation}           {translation}
  +{;} { \PrintReprint}               {reprint}
  +{.} { }                            {note}
  +{.} {}                             {transition}
  +{} { \PrintDOI}                   {doi}
  +{} { available at \url}            {eprint}
  +{}  {\SentenceSpace \PrintReviews} {review}
}

\usepackage{microtype}

\numberwithin{equation}{section}
\theoremstyle{plain}
\newtheorem{theorem}[equation]{Theorem}
\newtheorem{lemma}[equation]{Lemma}
\newtheorem{proposition}[equation]{Proposition}
\newtheorem{corollary}[equation]{Corollary}
\theoremstyle{definition}
\newtheorem{definition}[equation]{Definition}

\theoremstyle{remark}
\newtheorem{remark}[equation]{Remark}

\newcommand*{\s}{s}
\newcommand*{\rg}{r}
\newcommand*{\Hilm}[1][H]{\mathcal #1}
\newcommand{\idealin}{\mathrel{\triangleleft}} 

\newcommand*{\C}{\mathbb C}
\newcommand*{\Z}{\mathbb Z}
\newcommand*{\N}{\mathbb N}

\newcommand*{\T}{\mathbb T}
\newcommand*{\Bound}{\mathbb B}
\newcommand*{\Comp}{\mathbb K}
\newcommand*{\Mat}{\mathbb M}
\newcommand*{\Ideals}{\mathbb I}
\newcommand*{\Idealsc}{\mathbb I_\mathrm p}
\newcommand*{\Idealss}{\mathbb I_\mathrm s}
\newcommand*{\Open}{\mathbb O}
\newcommand*{\Openc}{\mathbb{CO}}

\newcommand*{\Mult}{\mathcal M}
\newcommand*{\Cont}{\mathrm C}
\newcommand*{\reg}{\mathrm{reg}}
\newcommand*{\Eext}{\mathfrak{Ext}}

\newcommand*{\id}{\textup{id}}
\newcommand*{\full}{\mathrm{f}}

\newcommand*{\KK}{\textup{KK}}
\newcommand*{\K}{\textup{K}}

\newcommand*{\prop}{\mathcal{P}}

\newcommand*{\ev}{\textup{ev}}
\newcommand*{\Cst}{\textup C^*}
\newcommand*{\Cstar}{\texorpdfstring{$\textup C^*$\nb-}{C*-}}
\newcommand*{\Star}{\texorpdfstring{$^*$\nb-}{*-}}

\newcommand*{\Fam}{\mathcal{F}}

\newcommand*{\U}{\mathcal U}

\newcommand*{\nb}{\nobreakdash}  
\newcommand*{\alb}{\hspace{0pt}} 

\mathtoolsset{mathic}
\DeclarePairedDelimiter{\abs}{\lvert}{\rvert}
\DeclarePairedDelimiter{\norm}{\lVert}{\rVert}
\DeclarePairedDelimiter{\ket}{\lvert}{\rangle}
\DeclarePairedDelimiter{\bra}{\langle}{\rvert}
\DeclarePairedDelimiterX{\braket}[2]{\langle}{\rangle}{#1\,\delimsize\vert\,\mathopen{}#2}

\DeclarePairedDelimiterX{\setgiven}[2]{\{}{\}}{#1\,{:}\,\mathopen{}#2}

\DeclareMathOperator{\supp}{supp}
\DeclareMathOperator{\Ext}{Ext}
\DeclareMathOperator{\Hom}{Hom}
\DeclareMathOperator{\Prim}{Prim}

\newcommand*{\congto}{\xrightarrow\sim}

\newcommand*{\defeq}{\mathrel{\vcentcolon=}}
\newcommand*{\into}{\rightarrowtail}
\newcommand*{\prto}{\twoheadrightarrow}

\begin{document}

\title{Correspondences between graph C*-algebras}

\author{Rasmus Bentmann}
\email{rasmusbentmann@gmail.com}

\author{Ralf Meyer}
\email{rmeyer2@uni-goettingen.de}

\address{Mathematisches Institut\\
  Georg-August Universit\"at G\"ottingen\\
  Bunsenstra\ss{}e 3--5\\
  37073 G\"ottingen\\
  Germany}

\begin{abstract}
  We describe proper correspondences from graph \Cstar{}algebras
  to arbitrary \Cstar{}algebras by \(\K\)\nb-theoretic data.  If
  the target \Cstar{}algebra is a graph \Cstar{}algebra as
  well, we may lift an isomorphism on a certain invariant to
  correspondences back and forth that are inverse to each other up
  to ideal-preserving homotopy.  This implies a classification
  theorem for purely infinite graph \Cstar{}algebras up to
  stable isomorphism.
\end{abstract}


\thanks{Supported by the DFG grant ``Classification of non-simple
  purely infinite C*-algebras.''  This work is part of the project
  Graph Algebras partially supported by EU grant
  HORIZON-MSCA-SE-2021 Project 101086394.}

\maketitle

\section{Introduction}
\label{sec:intro}

Graph \Cstar{}algebras form a class of \Cstar{}algebras
that are well-behaved in various ways.  For instance, they are
all nuclear and satisfy the Universal Coefficient Theorem.
Moreover, their ideal structure and \(\K\)\nb-theory is
well-understood in terms of the underlying graph, and the
\(\K\)\nb-theory has some special properties: the \(\K_1\)\nb-group
of a graph \Cstar{}algebra is always free, and the exponential map
associated to a (gauge-invariant) ideal in a graph \Cstar{}algebra
always vanishes.  In addition, every simple graph \Cstar{}algebra
is either approximately finite-dimensional or purely infinite.
(See \cite{Raeburn:Graph_algebras} for (references for) these
well known facts.)  Hence it is natural to explore what the
general classification theory of \Cstar{}algebras has to say
about \Cstar{}algebras with these features.  Among other
places, this has been done in \cites{MR2562779,
Eilers-Restorff-Ruiz:Strong_class_of_ext, Eilers-Tomforde:Graph,
MR3056712, MR3194162, MR2666426} for the mixed cases and in
\cites{Meyer-Nest:Filtrated_K, Arklint-Bentmann-Katsura:Reduction,
Bentmann:Real_rank_zero_and_int_cancellation,
Bentmann-Meyer:More_general} for the (strongly) purely infinite case.
While many results have been proved, a classification of nonunital graph
\Cstar{}algebras along these lines is not yet in reach.

The approach to classification using the symbolic dynamics associated
to the underlying graph has been very successful recently, culminating
in the complete classification of unital graph \Cstar{}algebras
(see~\cite{Eilers-Restorff-Ruiz-Sorensen:Complete} and the references
therein, and~\cites{Restorff:Classification, Rordam:Class_of_CK} for
earlier results).  Unitality is, however, a strong restriction for
graph \Cstar{}algebras because the unitisation of a graph
\Cstar{}algebra need not be a graph \Cstar{}algebra.  A graph
\Cstar{}algebra is unital if and only if the underlying graph has
finitely many vertices.  It seems difficult to generalise
the ideas of~\cite{Eilers-Restorff-Ruiz-Sorensen:Complete} to the
nonunital case.

The present article introduces a new approach to classify graph
\Cstar{}algebras.  The goal is to understand (proper)
\Cstar{}correspondences out of graph \Cstar{}algebras well enough to
be able to establish existence and uniqueness results for proper
\Cstar{}correspondences lifting isomorphisms on suitable
invariants.  We achieve this with uniqueness up to
(ideal-preserving) homotopy.  This is the main technical
achievement of the article.  Uniqueness up to homotopy is, however, too
weak for Elliott's Approximate Intertwining Argument.  Hence our
results do not suffice for a general classification of graph
\Cstar{}algebras up to stable isomorphism.  For purely infinite graph
\Cstar{}algebras, however, it implies an equivalence in Kirchberg's
KK-theory, which is enough to apply Kirchberg's Classification
Theorem~\cite{Kirchberg:Michael}.  So we get a classification result
for all purely infinite (possibly nonunital) graph \Cstar{}algebras
up to stable isomorphism.  In general, our method only gives
a classification theorem for graph \Cstar{}algebras up to
ideal-preserving homotopy equivalence.

We now describe the sections of this article.  In
Section~\ref{sec:corr_vs_homomorphism}, we collect some basic
results about proper \Cstar{}correspondences, their induced maps on
ideal lattices and their relation to \Star{}homomorphisms into
stabilised \Cstar{}algebras.  Section~\ref{sec:corr_out_of_graph}
studies proper \Cstar{}correspondences out of graph
\Cstar{}algebras.  The main result says that any homomorphism
between the monoids of equivalence classes of projections lifts to a
proper \Cstar{}correspondence if the domain is a graph
\Cstar{}algebra.  Section~\ref{sec:Proj_to_ideal-K} relates the
monoid of equivalence classes of projections to the ideal lattice
and the ordered \(\K_0\)\nb-groups of ideals, especially in the case
of graph \Cstar{}algebras.  Section~\ref{sec:diagrams} discusses
some homological algebra for certain diagrams of Abelian groups,
especially the computation of the Ext groups, which is where the
obstruction class and other homological invariants of
correspondences live.
Section~\ref{sec:add_K1} adds
information related to~\(\K_1\) to the picture.  We show that a pair
of maps on equivalence classes of projections and on~\(\K_1\) is
realised by a proper \Cstar{}correspondence if and only if a certain
variant of the obstruction class
of~\cite{Bentmann-Meyer:More_general} vanishes.
In Section~\ref{sec:purely_inf}, we combine
our new results with Kirchberg's Classification Theorem to classify
purely infinite graph \Cstar{}algebras.  In Sections
\ref{sec:homotopy_corres} and~\ref{sec:circles}, we allow arbitrary
graph \Cstar{}algebras again and prove an ideal-preserving
homotopy equivalence instead of an isomorphism under the same
assumptions as in Section~\ref{sec:purely_inf}.  More precisely,
Section~\ref{sec:homotopy_corres} only proves a homotopy equivalence
preserving gauge-invariant ideals.  Section~\ref{sec:circles}
contains some extra arguments that allow to prove an
ideal-preserving homotopy equivalence for all ideals,
gauge-invariant or not.

The authors thank the referee for carefully reading the article and
spotting an error in an earlier version of this article.

\section{Correspondences versus homomorphisms into stabilisations}
\label{sec:corr_vs_homomorphism}

This section contains results for general \(\sigma\)\nb-unital
\Cstar{}algebras,
which we are going to specialise to graph \Cstar{}algebras
in later sections.  Let \(A\)
and~\(B\)
be two \Cstar{}algebras.  We are going to show that proper
\Cstar{}correspondences
between \(A\)
and~\(B\)
are ``equivalent'' to \Star{}homomorphisms from~\(A\)
to the \Cstar{}stabilisation
\(B\otimes\Comp\)
when~\(A\)
is \(\sigma\)\nb-unital.
Here \(\Comp\defeq \Comp(\ell^2\N)\)
is the \Cstar{}algebra
of compact operators on a separable Hilbert space.  If \(A\)
is \Cstar{}stable
and both \(A\)
and~\(B\)
are \(\sigma\)\nb-unital,
then full, proper correspondences are equivalent to nondegenerate
\Star{}homomorphisms \(A\to B\otimes\Comp\).
Let \(\Ideals(A)\)
be the lattice of closed ideals of~\(A\).
A \Star{}homomorphism
\(\varphi\colon A\to B\otimes\Comp\)
induces a map \(\Ideals(\varphi)\colon \Ideals(A)\to\Ideals(B)\).
We describe it in terms of the proper \Cstar{}correspondence
corresponding to~\(\varphi\).
All this is either easy or follows from the proofs by Mingo and
Phillips~\cite{Mingo-Phillips:Triviality} of the Kasparov
Stabilisation Theorem and the Brown--Green--Rieffel Theorem.

First we recall some definitions.  An \emph{\(A,B\)-correspondence}
is a (right) Hilbert \(B\)\nb-module~\(\Hilm\)
with a nondegenerate \Star{}homomorphism~\(\varphi\)
from~\(A\)
to the \Cstar{}algebra \(\Bound(\Hilm)\)
of adjointable operators on~\(\Hilm\).
It is \emph{proper} if~\(\varphi(A)\)
is contained in the \Cstar{}algebra~\(\Comp(\Hilm)\)
of compact operators on~\(\Hilm\).
Let \(\braket{\Hilm}{\Hilm}\idealin B\)
be the ideal generated by the inner products \(\braket{x}{y}\)
for \(x,y\in\Hilm\).
Call~\(\Hilm\)
\emph{full} if \(\braket{\Hilm}{\Hilm}=B\).
An \emph{isomorphism} between two \(A,B\)-correspondences
\(\Hilm_1\)
and~\(\Hilm_2\)
is a unitary operator \(u\colon \Hilm_1 \congto \Hilm_2\)
of right Hilbert \(B\)\nb-modules
that intertwines the left actions of~\(A\),
that is, \(u(a\cdot\xi) = a\cdot(u\xi)\)
for all \(a\in A\),
\(\xi\in\Hilm_1\).
The \emph{direct sum} \(\Hilm_1\oplus\Hilm_2\)
of two correspondences is defined in the obvious way.  It is again
proper if \(\Hilm_1\) and~\(\Hilm_2\) are proper.  And
\begin{equation}
  \label{eq:ideal_for_sum}
  \braket{\Hilm_1\oplus\Hilm_2}{\Hilm_1\oplus\Hilm_2}
  = \braket{\Hilm_1}{\Hilm_1} + \braket{\Hilm_2}{\Hilm_2},
\end{equation}
so that \(\Hilm_1\oplus\Hilm_2\)
is full if \(\Hilm_1\)
or~\(\Hilm_2\)
is full.  The direct sum operation descends to isomorphism classes, is
associative, and has the zero element~\(\{0\}\).
Thus isomorphism classes of proper \(A,B\)-correspondences
with the operation~\(\oplus\)
induced by the direct sum of correspondences form a monoid.  We denote
it by \(\prop(A,B)\).

An \(A,B\)-correspondence~\(\Hilm\)
induces a map \(\Ideals(\Hilm)\colon \Ideals(A)\to\Ideals(B)\)
by mapping \(I\idealin A\)
to the ideal \(\braket{I\Hilm}{I\Hilm}\idealin B\)
generated by the inner products \(\braket{ax}{by}\)
with \(a,b\in I\), \(x,y\in \Hilm\).  Since \(I^*\cdot I=I\), we have
\[
\braket{I\Hilm}{I\Hilm}
= \braket{\Hilm}{I\Hilm}
= \braket{I\Hilm}{\Hilm}.
\]
Therefore, the map \(\Ideals(\Hilm)\)
commutes with arbitrary suprema, that is,
\[
\Ideals(\Hilm)\biggl(\bigvee_{I\in S} I\biggr)
= \biggl(\bigvee_{I\in S} \Ideals(\Hilm)(I) \biggr)
\]
for any set~\(S\)
of ideals in~\(A\),
where \(\bigvee_{I\in S} I = \overline{\sum_{I\in S} I}\)
is the supremum of~\(S\)
in~\(\Ideals(A)\).
As a consequence, the map \(\Ideals(\Hilm)\)
is monotone and preserves the zero ideals.  It need not commute with
intersections, however.  And
\(\Ideals(\Hilm)(A) = \braket{\Hilm}{\Hilm}\)
is only equal to~\(B\) if~\(\Hilm\) is full.

Let \(\Hilm^\infty \defeq \Hilm\otimes \ell^2(\N)\)
or, equivalently, the direct sum of countably many copies
of~\(\Hilm\).
There is an isomorphism
\(\Comp(\Hilm)\otimes\Comp \congto \Comp(\Hilm^\infty)\),
which maps \(S\otimes T \in \Comp(\Hilm)\otimes\Comp\)
to the unique operator \(S\otimes T\)
on \(\Hilm\otimes \ell^2\N\)
that satisfies \((S\otimes T)(x\otimes y) \defeq S(x) \otimes T(y)\)
for all \(x\in \Hilm\),
\(y\in \ell^2\N\).
In particular, \(B\otimes\Comp \cong \Comp(B^\infty)\)
because \(B\cong \Comp(B)\)
when we view~\(B\) as a Hilbert \(B\)\nb-module in the obvious way.

A \Star{}homomorphism \(\varphi\colon A\to B\otimes\Comp\)
induces a proper \(A,B\)-correspondence~\(\Hilm_\varphi\)
as follows.  The underlying Hilbert \(B\)\nb-module
is \(\Hilm\defeq \varphi(A)\cdot B^\infty \subseteq B^\infty\),
and~\(A\)
acts nondegenerately on~\(\Hilm\)
through~\(\varphi\).
If~\(\varphi\)
is nondegenerate, then \(\Hilm = B^\infty\)
and thus~\(\Hilm_\varphi\)
is full.  The next proposition reverses this construction if~\(A\)
is \(\sigma\)\nb-unital,
that is, if~\(A\) has a sequential approximate unit.

\begin{proposition}
  \label{pro:proper_corr_to_homomorphism}
  If~\(A\)
  is \(\sigma\)\nb-unital,
  then any proper \(A,B\)-correspondence~\(\Hilm\)
  is isomorphic to~\(\Hilm_\varphi\)
  for a \Star{}homomorphism \(\varphi\colon A\to B\otimes\Comp\).
  If \(A\)
  and~\(B\)
  are \(\sigma\)\nb-unital
  and \(A\)
  is \Cstar{}stable,
  then any proper, full \(A,B\)-correspondence~\(\Hilm\)
  is isomorphic to~\(\Hilm_\varphi\)
  for a nondegenerate \Star{}homomorphism
  \(\varphi\colon A\to B\otimes\Comp\).
\end{proposition}

\begin{proof}
  First assume only that~\(A\)
  is \(\sigma\)\nb-unital.
  Let \(\varphi_0\colon A\to\Comp(\Hilm)\)
  denote the left action on~\(\Hilm\).
  Since it is nondegenerate, it maps a sequential approximate unit
  in~\(A\)
  to one in~\(\Comp(\Hilm)\).
  The Hilbert module~\(\Hilm\)
  is countably generated if and only if~\(\Comp(\Hilm)\)
  is \(\sigma\)\nb-unital.
  So~\(\Hilm\)
  is countably generated.  Then the Kasparov Stabilisation Theorem
  gives
  a unitary operator \(\Hilm \oplus B^\infty \cong B^\infty\)
  and, in particular, an adjointable isometry
  \(S\colon \Hilm \to B^\infty\).
  Then \(\varphi(a) \defeq S \varphi_0(a) S^* \in \Comp(B^\infty)
  \cong B \otimes \Comp\)
  for \(a\in A\)
  defines a \Star{}homomorphism \(\varphi\colon A\to B\otimes\Comp\).
  By construction, \(\varphi(A)\cdot B^\infty = S(\Hilm)\)
  and \(S\)
  is an isomorphism of \(A,B\)-correspondences
  from~\(\Hilm\) to~\(\Hilm_\varphi\).

  Now assume, in addition, that \(A= A_0\otimes\Comp\)
  is \Cstar{}stable
  and that~\(B\)
  is \(\sigma\)\nb-unital.  The nondegenerate left action of~\(A\)
  on~\(\Hilm\)
  induces an action of \(\Comp\subseteq \Mult(A\otimes\Comp)\).
  Let \(p\in\Comp\)
  be a rank-\(1\)
  projection and \(\Hilm_0\defeq (1\otimes p)\Hilm\).
  The matrix units in~\(\Comp\)
  allow to identify \(\Hilm \cong \Hilm_0^\infty\).
  Since~\(\Hilm\)
  is full, so is~\(\Hilm_0\).
  Since~\(\Hilm\)
  is countably generated, so is~\(\Hilm_0\).
  Since \(B=\Comp(B)\)
  is \(\sigma\)\nb-unital,
  \(B\)
  is a countably generated Hilbert \(B\)\nb-module.
  Hence \cite{Mingo-Phillips:Triviality}*{Theorem~1.9} gives a
  unitary operator
  \(S\colon \Hilm \congto \Hilm_0^\infty \congto B^\infty\).
  Now the \Star{}homomorphism \(\varphi\colon A\to B\otimes \Comp\),
  \(a \mapsto S a S^*\), is nondegenerate,  
  and~\(S\) is an isomorphism from~\(\Hilm\) to~\(\Hilm_\varphi\).
\end{proof}

If \(A=\C\),
then a proper \(\C,B\)-correspondence
is just a Hilbert \(B\)\nb-module~\(\Hilm\)
with \(\id_{\Hilm}\in\Comp(\Hilm)\).
We call such Hilbert \(B\)\nb-modules
\emph{proper}.  Let
\(\Mat_\infty(B) = \bigcup_{n=0}^{\infty} \Mat_n(B)\)
denote the set of infinite matrices over~\(B\)
with finitely many nonzero entries.

\begin{theorem}
  \label{the:prop_CB}
  Let~\(B\)
  be a \Cstar{}algebra.  A Hilbert \(B\)\nb-module~\(\Hilm\)
  is proper if and only if there is a projection
  \(p\in\Mat_\infty(B)\)
  with \(\Hilm\cong p\cdot B^\infty\).
  Let \(p_1,p_2\in\Mat_\infty(B)\)
  be projections.  Unitary operators \(p_1 B^\infty \to p_2 B^\infty\)
  are in bijection with Murray--von Neumann
  equivalences between \(p_1\)
  and~\(p_2\),
  that is, \(u\in \Mat_\infty(B)\)
  with \(u u^* = p_2\) and \(u^* u = p_1\).
\end{theorem}

\begin{proof}
  This follows from well known results if~\(B\) is unital (see
  \cite{Wegge-Olsen:K-theory}*{Theorem~15.4.2}).  The statements
  that we claim still hold if~\(B\) is nonunital (see also
  \cite{Banerjee-Meyer:Coarse}*{Proposition~4.2}).  We include a
  proof here for the convenience of the reader.

  First assume that \(\Hilm\cong p B^m\)
  for a projection \(p\in\Mat_m(B)\).
  We let \(V\colon \Hilm\to B^m\)
  be the resulting isometric embedding.  It is adjointable with the
  adjoint \(B^m\to \Hilm\),
  \(x\mapsto p\cdot x\).
  The compositions \(V^*V\)
  and~\(VV^*\)
  are the identity map on~\(\Hilm\)
  and left multiplication by~\(p\),
  respectively.  The latter is a compact operator on~\(B^m\)
  because \(\Comp(B^m)=\Mat_m(B)\).
  Since~\(VV^*\)
  is compact, so is \(V=VV^*V\) and hence also \(V^*V=\id_{\Hilm}\).

  Conversely, assume that~\(\id_{\Hilm}\)
  is compact.  Then there is a finite-rank approximation
  \(\sum_{j=1}^m {}\ket{x_j}\bra{y_j}\) with, say,
  \[
  \norm[\bigg]{\sum_{j=1}^m {}\ket{x_j}\bra{y_j} -\id_{\Hilm}} <1.
  \]
  Then \(\sum_{j=1}^m {}\ket{x_j}\bra{y_j}\)
  is invertible in \(\Comp(\Hilm)\);
  let \(s\in \Comp(\Hilm)\) be its inverse.  Let
  \[
  y\defeq (\bra{y_j})_{j=1,\dotsc ,m}\colon \Hilm\to B^m,\qquad
  x\defeq (\ket{x_j})_{j=1,\dotsc,m}\colon B^m\to \Hilm;
  \]
  these are adjointable maps with \((s\circ x)\circ y=\id_{\Hilm}\).
  Then \(e = y\circ (s\circ x)\colon B^m\to B^m\)
  is idempotent, and adjointable as a composition of adjointable
  operators.  Hence its range is complementable.  The operators
  \(e\) and~\(y\) have the same range and~\(y^*\) is surjective
  because \((s\circ x)\circ y=\id_{\Hilm}\).
  As a consequence, the operator~\(y\)
  has a polar decomposition.  This gives an isometry
  \(V\colon \Hilm\to B^m\).
  Since~\(\id_{\Hilm}\)
  is compact, so is \(V=V\circ \id_{\Hilm}\)
  and hence~\(VV^*\).
  Thus~\(VV^*\)
  is left multiplication by some \(p\in \Mat_m(B)\).
  This is a projection whose image is isomorphic to~\(\Hilm\)
  via~\(V\).
\end{proof}

Thus \(\prop(B) \defeq \prop(\C,B)\)
is canonically isomorphic to the monoid of Murray--von Neumann
equivalence classes of projections in \(\Mat_\infty(B)\)
with the usual operation~\(\oplus\), even if~\(B\) is not unital.

Proper correspondences may be composed by tensor product.  This
descends to a category structure on isomorphism classes of proper
correspondences.  In particular, there is a canonical composition
\(\prop(A) \times \prop(A,B) \to \prop(B)\).
It is additive in both variables.  Equivalently, a proper
\(A,B\)-correspondence~\(\Hilm\)
induces a monoid homomorphism \(\Hilm_*\colon \prop(A)\to\prop(B)\),
\([\Hilm[M]] \mapsto [\Hilm[M] \otimes_A \Hilm]\),
and these homomorphisms satisfy
\((\Hilm_1)_* + (\Hilm_2)_* = (\Hilm_1\oplus\Hilm_2)_*\).

\medskip

Next, we define ``unitary equivalence'' of \Star{}homomorphisms
\(A\to B\otimes\Comp\) in such a way that it corresponds exactly to
unitary equivalence of the corresponding proper
\(A,B\)-correspondences.  Let
\(R_\varphi \defeq \varphi(A)\cdot B\otimes\Comp\).  This is a right
ideal.  We also view it as a Hilbert module over \(B\otimes\Comp\).

\begin{definition}
  Let \(\varphi_0,\varphi_1\colon A\rightrightarrows B\otimes\Comp\)
  be two \Star{}homomorphisms.  A \emph{unitary equivalence}
  from~\(\varphi_0\)
  to~\(\varphi_1\)
  is a unitary operator
  \(u\colon R_{\varphi_0} \congto R_{\varphi_1}\)
  with \(u\varphi_0(a)u^* = \varphi_1(a)\)
  for all \(a\in A\),
  as operators on the Hilbert \(B\)\nb-module~\(R_{\varphi_1}\).
\end{definition}

A unitary multiplier~\(\bar{u}\)
of~\(B\otimes\Comp\)
with \(u\varphi_0(a)u^* = \varphi_1(a)\)
in \(B\otimes\Comp\)
for all \(a\in A\)
restricts to a unitary operator
\(u\colon R_{\varphi_0} \congto R_{\varphi_1}\)
that is a unitary equivalence from~\(\varphi_0\)
to~\(\varphi_1\).
Thus we allow more unitary equivalences than usual.

The argument that identifies the multiplier algebra of
\(\Comp(\Hilm)\)
for a Hilbert \(B\)\nb-module~\(\Hilm\)
with~\(\Bound(\Hilm)\)
gives a bijection between unitary equivalences
\(\varphi_1 \to \varphi_2\)
and isomorphisms of correspondences
\(\Hilm_{\varphi_1} \to \Hilm_{\varphi_2}\).
Thus isomorphism classes of proper \(A,B\)-correspondences
are in bijection with unitary equivalence classes of
\Star{}homomorphisms \(A\to B\otimes\Comp\).
And isomorphism classes of proper, full \(A\otimes\Comp,B\)-correspondences
are in bijection with unitary equivalence classes of nondegenerate
\Star{}homomorphisms \(A\otimes\Comp\to B\otimes\Comp\),
for the standard notion of unitary equivalence.

A \Star{}homomorphism \(\varphi\colon A\to B\otimes\Comp\)
induces a map
\(\Ideals(\varphi)\colon \Ideals(A)\to\Ideals(B\otimes\Comp)\cong
\Ideals(B)\)
by mapping \(I\idealin A\)
to the ideal in~\(B\otimes\Comp\)
generated by \(\varphi(I)\subseteq B\otimes\Comp\).
A routine computation identifies this map with the map
\(\Ideals(\Hilm_\varphi)\) defined above.

For \Cstar{}algebra
classification, the equivalence relation of unitary equivalence for
\Star{}homomorphisms is too fine, even with our mild modification.  We
should use asymptotic or approximate unitary equivalence or
homotopy.  A \emph{homotopy}
from~\(\varphi_0\)
to~\(\varphi_1\)
is a \Star{}homomorphism
\(\Phi\colon A\to \Cont([0,1],B\otimes\Comp)\)
with \(\ev_t \circ \Phi = \varphi_t\)
for \(t=0,1\).
An \emph{asymptotic unitary equivalence} from~\(\varphi_0\)
to~\(\varphi_1\)
is a family of unitary operators
\(u_t\colon R_{\varphi_0} \congto R_{\varphi_1}\)
for \(t\in[0,\infty)\)
that is \Star{}strongly continuous in the sense that
\(t\mapsto u_t(x) \in R_{\varphi_1}\) and
\(t\mapsto u_t^*(y) \in R_{\varphi_0}\)
are norm-continuous for \(x\in R_{\varphi_0}\), \(y\in
R_{\varphi_1}\) and that satisfies
\[
\lim_{t\to\infty} u_t\varphi_0(a)u_t^* = \varphi_1(a)
\]
for all \(a\in A\),
as operators on the Hilbert \(B\)\nb-module~\(R_{\varphi_1}\).
An \emph{approximate unitary equivalence} from~\(\varphi_0\)
to~\(\varphi_1\)
is a family of unitary operators
\(u_n\colon R_{\varphi_0} \congto R_{\varphi_1}\)
for \(n\in\N\) that satisfies
\[
\lim_{n\to\infty} u_n\varphi_0(a)u_n^* = \varphi_1(a)
\]
for all \(a\in A\).
For proper \(A,B\)-correspondences, these three notions correspond
to the following:

\begin{definition}
  A \emph{homotopy} from~\(\Hilm_0\)
  to~\(\Hilm_1\)
  is an \(A,\Cont([0,1],B)\)-correspondence~\(\Hilm\)
  together with isomorphisms
  \(\Hilm\otimes_{\ev_t} B \cong \Hilm_t\) for \(t=0,1\).
  An \emph{asymptotic isomorphism} from~\(\Hilm_0\)
  to~\(\Hilm_1\)
  is a \Star{}strongly continuous family of unitary
  operators~\((u_t)_{t\in[0,\infty)}\)
  between the underlying Hilbert \(B\)\nb-modules that satisfies
  \[
  \lim_{t\to\infty} u_t (a\cdot_0 u_t^* x) = a \cdot_1 x
  \]
  for all \(a\in A\),
  \(x\in\Hilm_1\);
  here \(\cdot_i\)
  denotes the left action of~\(A\)
  on~\(\Hilm_i\).
  An \emph{approximate isomorphism} is defined similarly, replacing the
  indexing set \([0,\infty)\) by~\(\N\).
\end{definition}

\section{Correspondences out of a graph \texorpdfstring{$C^*$}{C*}-algebra}
\label{sec:corr_out_of_graph}

Let \(G = (\rg,\s\colon E\rightrightarrows V)\)
be a directed graph with sets of vertices and edges \(V\)
and~\(E\),
respectively, and range and source maps \(\rg\)
and~\(\s\).
The \emph{graph \Cstar{}algebra of~\(G\)}
is the universal \Cstar{}algebra \(\Cst(G)\)
generated by mutually orthogonal projections~\(p_v\)
for \(v\in V\)
and partial isometries~\(t_e\) for \(e\in E\) subject to the relations
\begin{enumerate}[label=(CK\arabic*)]
\item \label{en:CK1} \(t_e^* t_e = p_{\s(e)}\) for all \(e\in E\);
\item \label{en:CK3} \(\sum_{\rg(e)=v} t_e t_e^* = p_v\) for all \(v\in V\)
  with \(0<\abs{\rg^{-1}(v)}<\infty\);
\item \label{en:CK2} \(t_e t_e^* \le p_{\rg(e)}\) for all \(e\in E\)
  and \(t_e^* t_f = 0\) for all \(e,f\in E\) with \(e\neq f\).
\end{enumerate}
We call \(v\in V\) \emph{regular} if \(0<\abs{\rg^{-1}(v)}<\infty\);
let \(V_\reg\subseteq V\) be the subset of regular vertices.  The
graph~\(G\) is \emph{regular} if all vertices are regular.  Then the
relation~\ref{en:CK3} is imposed for all vertices.  More generally,
the graph~\(G\) is \emph{row-finite} if \(\abs{\rg^{-1}(v)}<\infty\)
for all \(v\in V\).  Then the relation~\ref{en:CK2} is redundant.
Any graph \Cstar{}algebra~\(\Cst(G)\) is stably isomorphic
to~\(\Cst(G')\) for a regular graph~\(G'\), which is constructed
in~\cite{Muhly-Tomforde:Adding_tails} by ``adding tails'' to~\(G\).
Since we only care about graph \Cstar{}algebras up to stable
isomorphism, we may as well assume~\(G\) to be regular.  Actually,
our methods work equally well for row-finite graphs.  We assume all
graphs in the following to be row-finite, sometimes tacitly.

\begin{theorem}[\cite{Ara-Moreno-Pardo:Nonstable_K_for_Graph_Algs}*{Theorems
    3.5 and~7.1}]
  \label{the:prop_graph}
  Let~\(G\) be a row-finite graph.  Each projection
  \(q\in \Mat_\infty(\Cst(G))\) is equivalent to a finite sum
  \(\sum_{v\in V} c_v p_v\) of vertex projections with
  \((c_v)_{v\in V}\in \N[V]\), that is, \(c_v\in \N\) for all
  \(v\in V\) and \(c_v\neq0\) only for finitely many \(v\in V\).
\end{theorem}

\begin{theorem}
  \label{the:lift_prop_CB}
  Let~\(G\) be a row-finite graph and~\(B\) a \Cstar{}algebra.
  Proper \(\Cst(G),B\)-correspondences are equivalent to
  families \(\bigl((\Hilm_v)_{v\in V},(U_v)_{v\in V_\reg}\bigr)\),
  where each~\(\Hilm_v\)
  is a proper Hilbert \(B\)\nb-module and~\(U_v\)
  is a unitary operator
  \(\bigoplus_{e\in E^v} \Hilm_{\s(e)} \congto \Hilm_v\).
  Here ``equivalent'' means an equivalence of categories, where
  arrows from \(\bigl((\Hilm_{1,v}),(U_{1,v})\bigr)\)
  to \(\bigl((\Hilm_{2,v}),(U_{2,v})\bigr)\)
  are families of unitaries
  \(W_v\colon \Hilm_{1,v} \congto \Hilm_{2,v}\) for \(v\in V\)
  that make the following diagrams commute for all \(v\in V_\reg\):
  \begin{equation}
    \label{eq:isom_corr_HU}
    \begin{tikzpicture}[baseline=(current bounding box.west)]
      \matrix (m) [cd,column sep=4em] {
        \bigoplus_{e\in E^v} \Hilm_{1,\s(e)} &
        \Hilm_{1,v}\\
        \bigoplus_{e\in E^v} \Hilm_{2,\s(e)} &
        \Hilm_{2,v}\\
      };
      \draw[cdar] (m-1-1) -- node {\(U_{1,v}\)} (m-1-2);
      \draw[cdar] (m-2-1) -- node {\(U_{2,v}\)} (m-2-2);
      \draw[cdar] (m-1-1) -- node[swap]
      {\(\bigoplus_{e\in E^v} W_{\s(e)}\)} (m-2-1);
      \draw[cdar] (m-1-2) -- node {\(W_v\)} (m-2-2);
    \end{tikzpicture}
  \end{equation}
  For any monoid homomorphism
  \(\varphi\colon \prop(\Cst(G))\to\prop(B)\),
  there is a proper
  \(\Cst(G),B\)-correspondence~\(\Hilm\) with \(\Hilm_*=\varphi\).
\end{theorem}

\begin{proof}
  Fix \(v\in V_\reg\).
  Let \(E^v \defeq \{e\in E\mid \rg(e)=v\}\)
  and let \(m_v = \abs{E^v}\)
  be the number of edges with range~\(v\).
  Totally order these edges, and let
  \(u_v\in \Mat_{1,m_v}(\Cst(G))\)
  be the row matrix with entries~\(t_e\)
  for \(e\in E^v\).
  The relations \ref{en:CK1} and~\ref{en:CK3} say that
  \(u_v u_v^* = \sum_{e\in E^v} t_e t_e^* = p_v\)
  and
  \(u_v^* u_v = (t_e^* t_f)_{e,f\in E^v} = (\delta_{e,f}
  p_{\s(e)})_{e,f\in E^v}\),
  the orthogonal direct sum \(\bigoplus_{e\in E^v} p_{\s(e)}\)
  of the projections~\(p_{\s(e)}\)
  for \(e\in E^v\).
  That is, \(u_v\)
  is a partial isometry with range and source projections \(p_v\)
  and \(\bigoplus_{e\in E^v} p_{\s(e)}\).  Since any vertex in
  \(\rg(E)\) is regular, the unitaries~\(u_v\) for \(v\in V_\reg\)
  contain the partial isometries~\(t_e\) for all edges.

  Now we prove that a proper \(\Cst(G),B\)-correspondence~\(\Hilm\)
  is equivalent to a family \(\bigl((\Hilm_v)_{v\in V},(U_v)_{v\in
    V_\reg}\bigr)\).
  Let \(\Hilm_v \defeq p_v\cdot\Hilm\subseteq \Hilm\).
  These are proper Hilbert \(B\)\nb-modules
  because \(p_v\in\Comp(\Hilm)\)
  for a proper \(\Cst(G),B\)-correspondence.
  The partial isometry~\(u_v\)
  for \(v\in V_\reg\)
  acts by a unitary operator
  \(U_v\colon \bigoplus_{e\in E^v} \Hilm_{\s(e)} \congto \Hilm_v\).
  This maps~\(\Hilm\)
  to a family \(((\Hilm_v),(U_v))\).
  Let \(\Hilm_1\) and~\(\Hilm_2\) be proper
  \(\Cst(G),B\)-correspondences and let
  \(W\colon \Hilm_1 \to\Hilm_2\) be an isomorphism.  This restricts
  to isomorphisms of Hilbert \(B\)\nb-modules \(W_v\colon
  \Hilm_{1,v} \to \Hilm_{2,v}\) because~\(W\) intertwines the left
  action of the projections~\(p_v\).  The
  diagrams~\eqref{eq:isom_corr_HU} commute because~\(W\) intertwines
  the left action of the generators~\(t_e\) for \(e\in E\) and hence
  of the matrices~\(u_v\) for \(v\in V_\reg\).

  Conversely, let~\(\Hilm_v\)
  for \(v\in V\)
  be proper Hilbert \(B\)\nb-modules
  and let
  \[
    U_v\colon \bigoplus_{e\in E^v} \Hilm_{\s(e)} \congto \Hilm_v
  \]
  be unitaries for \(v\in V_\reg\).
  Let \(\Hilm \defeq \bigoplus_{v\in V} \Hilm_v\).
  Let \(P_v\in\Comp(\Hilm)\)
  be the projection onto the summand~\(\Hilm_v\).
  Let~\(T_e\)
  for \(e\in E^v\)
  be the restriction of~\(U_v\) to the summand~\(\Hilm_{\s(e)}\).
  The operators \(P_v\)
  and~\(T_e\)
  on~\(\Hilm\)
  satisfy the Cuntz--Krieger relations \ref{en:CK1} and~\ref{en:CK3}.
  So there is a unique \Star{}homomorphism \(\Cst(G)\to\Comp(\Hilm)\)
  mapping~\(p_v\) to~\(P_v\) and~\(t_e\) to~\(T_e\).
  This is nondegenerate because it maps
  \(\bigl(\bigoplus_{v\in F} p_v\bigr)_F\)
  to an approximate unit in~\(\Comp(\Hilm)\).
  Thus the family \(((\Hilm_v),(U_v))\)
  comes from a proper \(\Cst(G),B\)-correspondence.  Let
  \(((\Hilm_{1,v}),(U_{1,v}))\)
  and \(((\Hilm_{2,v}),(U_{2,v}))\)
  be two such families and let~\((W_v)_{v\in V}\) be a family of
  unitaries \(W_v\colon \Hilm_{1,v} \congto \Hilm_{2,v}\)
  making the diagrams~\eqref{eq:isom_corr_HU} commute for all
  \(v\in V_\reg\).
  Then \(W\defeq \bigoplus_{v\in V} W_v\) is a unitary
  \(\bigoplus_{v\in V} \Hilm_{1,v} \to \bigoplus_{v\in V}
  \Hilm_{2,v}\), which intertwines the left actions of the
  generators~\(p_v\) for \(v\in V\) and~\(t_e\) for \(e\in E\).

  The functors from the category of proper
  \(\Cst(G),B\)-correspondences
  to the category of families \(((\Hilm_v),(U_v))\) and back
  are easily seen to be inverse to each other up to natural
  isomorphism.  That is, they form an equivalence of categories.

  Now let \(\varphi\colon \prop(\Cst(G))\to\prop(B)\) be
  a homomorphism.  We construct a family \(((\Hilm_v),(U_v))\)
  from it.  For \(v\in V\),
  let~\(\Hilm_v\)
  be a proper Hilbert \(B\)\nb-module
  that represents \(\varphi[p_v]\in \prop(B)\).
  Let \(v\in V_\reg\).  The equivalence~\(u_v\)
  in \(\Mat_\infty(\Cst(G))\)
  implies that \(\sum_{e\in E^v} {}[p_{\s(e)}] = [p_v]\)
  holds in \(\prop(\Cst(G))\).
  Since~\(\varphi\)
  is a homomorphism, this implies
  \(\sum_{e\in E^v} \varphi[p_{\s(e)}] = \varphi[p_v]\).
  Equivalently, there is a unitary~\(U_v\)
  between the corresponding proper Hilbert \(B\)\nb-modules
  \(\bigoplus_{e\in E^v} \Hilm_{\s(e)}\) and~\(\Hilm_v\).
  The Hilbert modules~\(\Hilm_v\)
  for \(v\in V\)
  and the unitaries~\(U_v\)
  for \(v\in V_\reg\)
  describe a proper \(\Cst(G),B\)-correspondence~\(\Hilm\),
  which induces a map \(\Hilm_*\colon \prop(\Cst(G)) \to \prop(B)\).
  We claim that \(\Hilm_*=\varphi\).
  By construction, \(\Hilm_*[p_v] = [\Hilm_v] = \varphi[p_v]\)
  for all \(v\in V\).
  The subset of~\(\prop(\Cst(G))\)
  where \(\Hilm_*\)
  and~\(\varphi\)
  coincide is a submonoid.  Hence it contains all finite sums of the
  vertex projections.  Any class in \(\prop(\Cst(G))\)
  is such a finite sum by Theorem~\ref{the:prop_graph}.  Hence
  \(\Hilm_*=\varphi\) holds on all of \(\prop(\Cst(G))\).
\end{proof}

\begin{remark}
  \label{rem:replace_Hilbi_by_projections}
  By Theorem~\ref{the:prop_CB}, any proper Hilbert \(B\)\nb-module
  is isomorphic to \(p\cdot B^\infty\) for some projection
  \(p\in\Mat_\infty(B)\), and unitary operators between proper
  Hilbert \(B\)\nb-modules are equivalent to equivalences of
  projections.  Hence we may replace the families
  \(\bigl((\Hilm_v)_{v\in V},(U_v)_{v\in V_\reg}\bigr)\) by families
  \(\bigl((P_v)_{v\in V},(U_v)_{v\in V_\reg}\bigr)\) consisting of
  projections in \(\Mat_\infty(B)\) and equivalences between them.
\end{remark}

Let \(\varphi\colon \prop(\Cst(G))\to\prop(B)\)
be a homomorphism.  There may be several proper
\(\Cst(G),B\)-correspondences~\(\Hilm\) with \(\Hilm_*=\varphi\).
By Theorem~\ref{the:lift_prop_CB}, \(\Hilm\) is given by proper
Hilbert \(B\)\nb-modules~\(\Hilm_v\)
for \(v\in V\)
and unitaries
\(U_v\colon \bigoplus_{e\in E^v} \Hilm_{\s(e)} \to \Hilm_v\)
for \(v\in V_\reg\).
We must have \([\Hilm_v] = \varphi[p_v]\) in \(\prop(B)\), which
determines~\(\Hilm_v\) up to isomorphism.  Given any family of
Hilbert \(B\)\nb-module isomorphisms \(W_v\colon \Hilm_v \to
\Hilm_v'\) for \(v\in V\),
there are unique unitaries~\(U_v'\)
for \(v\in V_\reg\)
that make~\eqref{eq:isom_corr_HU} commute, so that~\((W_v)\)
induces an isomorphism \((\Hilm_v,U_v) \cong (\Hilm'_v,U'_v)\).
Thus it does not matter how we choose the Hilbert
\(B\)\nb-modules~\(\Hilm_v\)
in their isomorphism classes.  But the choice of the operators~\(U_v\)
does matter.  They contain information about the action of the
\Cstar{}correspondence
on~\(\K_1\).
We shall make this more precise later.  Here we prepare this with some
simple observations.

\begin{lemma}
  \label{lem:difference_unitaries}
  Let~\(\Hilm_v\) for \(v\in V\) be Hilbert \(B\)\nb-modules.  Let
  \(U_v\colon \bigoplus_{e\in E^v} \Hilm_{\s(e)} \congto \Hilm_v\)
  for \(v\in V_\reg\) be unitaries.
  Families of unitaries
  \(U_v'\colon \bigoplus_{e\in E^v} \Hilm_{\s(e)} \congto \Hilm_v\)
  for \(v\in V_\reg\)
  are in bijection with families of unitaries
  \(\Upsilon_v\in \U(\Hilm_v)\)
  for \(v\in V_\reg\),
  by mapping~\(U'_v\)
  to \(\Upsilon_v\defeq U'_v U_v^*\in \U(\Hilm_v)\)
  and \(\Upsilon_v\in\U(\Hilm_v)\)
  to \(U'_v\defeq \Upsilon_v U_v\).
  An isomorphism between the \(\Cst(G),B\)-correspondences
  described by \(((\Hilm_v),(U_v))\)
  and \(((\Hilm_v),(\Upsilon_v U_v))\)
  is a family of unitaries \(W_v\in \U(\Hilm_v)\) for \(v\in V\) with
  \[
    \Upsilon_v = W_v \circ U_v \circ
    \biggl( \bigoplus_{e\in E^v} W_{\s(e)}^*\biggr)\circ U_v^*
  \]
  for all \(v\in V_\reg\).
\end{lemma}

\begin{proof}
  The bijective correspondence between \(U'_v\)
  and~\(\Upsilon_v\)
  is trivial.  The condition
  \(\Upsilon_v = W_v \circ U_v \circ \left( \bigoplus_{e\in E^v}
    W_{\s(e)}\right)^* \circ U_v^*\)
  is equivalent to the commutativity of the
  diagram~\eqref{eq:isom_corr_HU} with \(U_{1,v} = U_v\)
  and \(U_{2,v} = \Upsilon_v U_v\).
\end{proof}

Lemma~\ref{lem:difference_unitaries} allows us to describe
homotopies and asymptotic and approximate isomorphisms between
\(\Cst(G),B\)-correspondences.  Since homotopic projections are
equivalent, two correspondences that are homotopic must have
isomorphic Hilbert modules~\(\Hilm_v\).  The same holds for
asymptotically or approximately isomorphic correspondences.  Assume
for simplicity that they even have the same Hilbert
modules~\(\Hilm_v\).  By Lemma~\ref{lem:difference_unitaries}, the
two correspondences are described by unitaries
\(U_v,\Upsilon_v U_v\colon \bigoplus_{e\in E^v} \Hilm_{\s(e)}
\rightrightarrows \Hilm_v\) for some \(\Upsilon_v\in\U(\Hilm_v)\)
for \(v\in V_\reg\), respectively.  A homotopy between them consists
of norm-continuous homotopies \((\Upsilon_{v,t})_{t\in [0,1]}\) in
\(\U(\Hilm_v)\) for all \(v\in V_\reg\) and unitaries
\(W_{v,j}\in \U(\Hilm_v)\) for \(j=0,1\) and \(v\in V\), such that
\begin{align*}
  \Upsilon_{v,0}
  &= W_{v,0} \circ U_v \circ
    \biggl( \bigoplus_{e\in E^v} W_{\s(e),0}^*\biggr)\circ U_v^*,
  \\
  \Upsilon_{v,1}
  &= W_{v,1} \circ \Upsilon_v U_v \circ
    \biggl( \bigoplus_{e\in E^v} W_{\s(e),1}^*\biggr)\circ U_v^*.
\end{align*}
In other words, \((\Upsilon_{v,t})_{v\in V_\reg}\) for
\(t\in [0,1]\) yields a homotopy between two correspondences that
are isomorphic to those corresponding to \((U_v)_{v\in V_\reg}\) and
\((\Upsilon_v U_v)_{v\in V_\reg}\).

An asymptotic isomorphism between the two correspondences is given by
a family of norm-continuous paths of \(W_{v,t} \in \U(\Hilm_v)\)
for \(t\in [0,\infty\mathclose[\), \(v\in V\), such that
\[
  \lim_{t\to\infty} W_{v,t} \circ U_v \circ
  \biggl( \bigoplus_{e\in E^v} W_{\s(e),t}^*\biggr)\circ U_v^*
  = \Upsilon_v
\]
for all \(v\in V_\reg\);
there is no constraint on~\(W_{v,0}\).
An approximate isomorphism between them is similar, but instead
\(W_{v,t} \in \U(\Hilm_v)\)
is needed only for \(t\in \N\), \(v\in V\).  We can make this even
more explicit:

\begin{lemma}
  \label{lem:approximate_isomorphism}
  Let~\(V\) be countable and let \(((\Hilm_v),(U_v))\) and
  \(((\Hilm_v),(\Upsilon_v U_v))\) describe two
  \(\Cst(G),B\)-correspondences.  They are approximately isomorphic
  if and only if for each finite subset \(F\subseteq V\) and each
  \(\varepsilon>0\) there is a family of unitaries
  \(W_v\in \U(\Hilm_v)\) for \(v\in V\) with
  \[
    \norm*{\Upsilon_v - W_v \circ U_v \circ
    \biggl( \bigoplus_{e\in E^v} W_{\s(e)}^*\biggr)\circ U_v^*} <\varepsilon
  \]
  for all \(v\in F \cap V_\reg\).
\end{lemma}

\begin{proof}
  Write \(V= \bigcup_{n\in\N} F_n\) as an increasing union of finite
  subsets.  Let the unitaries~\(W_{v,n}\) be as in the statement of
  the lemma for \(F_n\subseteq V\) and \(\varepsilon=2^{-n}\).
  These unitaries give the desired approximate isomorphism.
\end{proof}

\subsection{The AF-case}
\label{sec:AF-case}

In this short subsection, we show that the techniques above already
suffice to classify \Cstar{}correspondences from \(\Cst(G)\)
to~\(B\) if our graph \Cstar{}algebra is AF.  This should be
expected because the following sections are mainly about
incorporating~\(\K_1(\Cst(G))\) into the picture, which is trivial
in the AF-case.

\begin{definition}
  A \emph{cycle} in a directed graph~\(G\) is a closed path
  \(e_1,\dotsc,e_n\), that is, \(s(e_1) = r(e_n)\),
  \(s(e_2) = r(e_1)\), \ldots, \(s(e_n) = r(e_{n-1})\).  A cycle is
  \emph{simple} if the vertices \(s(e_1),\dotsc,s(e_n)\) are all
  distinct.
\end{definition}

Assume for some time that~\(G\) has no cycles.  Equivalently,
\(\Cst(G)\) is an AF-algebra.  Then the invariant \(\prop(\Cst(G))\)
determines \(\Cst(G)\) up to isomorphism.  We are going to prove
that proper \(\Cst(G),B\)-correspondences are classified by the
induced map \(\prop(\Cst(G)) \to \prop(B)\) up to approximate
isomorphism:

\begin{proposition}
  \label{pro:AF-case}
  Let~\(G\) be a directed graph without cycles.  Let~\(B\) be a
  \Cstar{}algebra.  Identify proper
  \(\Cst(G),B\)-correspondences with families
  \(((\Hilm_v)_{v\in V},(U_v)_{v\in V_\reg})\).  Assume
  that \(((\Hilm_v),(U_v))\) and \(((\Hilm'_v),(U'_v))\) induce the
  same map \(\prop(\Cst(G)) \to \prop(B)\).
  \begin{enumerate}
  \item Let \(F\subseteq V_\reg\) be a finite subset.  Then
    \(((\Hilm'_v),(U'_v))\) is isomorphic to a proper
    \(\Cst(G),B\)-correspondence \(((\Hilm''_v),(U''_v))\) with
    \(\Hilm''_v=\Hilm_v\) for all \(v\in V\) and \(U''_v=U_v\) for
    \(v\in F\).
  \item If~\(G\) is countable, then \(((\Hilm_v),(U_v))\) and
    \(((\Hilm'_v),(U'_v))\) are approximately isomorphic.
  \end{enumerate}
\end{proposition}

\begin{proof}
  We have already explained how to arrange \(\Hilm'_v=\Hilm_v\) for
  all \(v\in V\).  Define \(x\prec y\) for \(x,y\in V\) if there is
  a path from~\(x\) to~\(y\) in~\(G\).  This is a partial
  order on~\(V\) because~\(G\) has no cycles.  Let
  \(F\subseteq V_\reg\) be finite.  We may enumerate
  \(F\cong \{1,\dotsc,n\}\) so that \(x\prec y\) implies \(x\le y\)
  in \(\{1,\dotsc,n\}\).  We construct the isomorphism from
  \(((\Hilm'_v),(U'_v))\) to \(((\Hilm''_v),(U''_v))\) by induction
  on~\(n\), the number of points in~\(F\).  We assume that we have
  such an isomorphism for \(\{1,\dotsc,n-1\}\).  Since isomorphisms
  of correspondences may be composed, we may assume that already
  \(U_i'=U_i\) for \(i=1,\dotsc,n-1\).  Now we define an isomorphism
  \((W_v)_{v\in V}\) by \(W_n\defeq U_n (U_n')^*\) and
  \(W_v\defeq 1\) for all \(v\neq n\) and, in particular, for all
  \(v\in V\setminus F\).  By
  Lemma~\ref{lem:approximate_isomorphism}, this is an isomorphism
  from \(((\Hilm'_v),(U'_v))\) to \(((\Hilm'_v),(U''_v))\) with
  \[
    U_v'' \defeq W_v\circ U_v'\circ \biggl(\bigoplus_{r(e)=v}
    W_{s(e)}^*\biggr).
  \]
  If \(v<n\), then \(s(e)<n\) for all \(e\in E\) with \(r(e)=v\), so
  that \(U_v''=U_v' = U_v\).  If \(v=n\), then \(s(e)<n\) for all
  \(e\in E\) with \(r(e)=n\), so that \(U_n''=W_n U_n' = U_n\).
  This proves the first statement.
  If~\(G\) is countable, we may enumerate \(V_\reg\cong\N\).  The
  isomorphisms \(((\Hilm_v),(U_v')) \cong ((\Hilm_v),(U_v''))\) for
  the finite subsets \(\{1,\dotsc,n\}\subseteq V\) form an
  approximate isomorphism between the \(\Cst(G),B\)-correspondences
  \(((\Hilm_v),(U_v'))\) and \(((\Hilm_v),(U_v))\) because
  \Star{}polynomials in the generators \(p_v,s_e\) are dense
  in~\(\Cst(G)\).  So the first statement implies the second one.
\end{proof}

\section{Projections versus K-theory diagrams}
\label{sec:Proj_to_ideal-K}

For a \Cstar{}algebra~\(B\)
that satisfies ``stable weak cancellation,''
we are going to describe~\(\prop(B)\)
using the ideal structure
of~\(B\)
and the \(\K_0^+\)\nb-monoids
for certain ideals of~\(B\).
Graph \Cstar{}algebras
and purely infinite \Cstar{}algebras
have stable weak cancellation.  This clarifies the relationship
between~\(\prop(B)\)
and invariants that have already been used in the classification of
nonsimple \Cstar{}algebras.

Let \(p\in\Mat_\infty(B)\)
be a projection.  The ideal \(\supp p\)
generated by~\(p\)
is the closed ideal generated by products~\(b_1^* p b_2\)
with \(b_1,b_2\in B^\infty\).
This is equal to the ideal generated by inner products of vectors in
the Hilbert module \(p \cdot B^\infty\)
and to the ideal generated by the matrix coefficients of~\(p\).
Any projection \(p\in\prop(B)\)
belongs to \(\Mat_\infty(\supp p)\)
and is a full projection there by construction of~\(\supp p\).
Equivalent projections generate the same ideal.  Hence we get a map
\[
\supp\colon \prop(B) \to \Ideals(B).
\]
It is \emph{additive}, that is, \(\supp 0 = 0\) and
\begin{equation}
  \label{eq:supp_additive}
  \supp (p\oplus q) = \supp p + \supp q  
\end{equation}
for all \(p,q\in\prop(B)\).
Hence it is also \emph{monotone}:
\begin{equation}
  \label{eq:supp_monotone}
  p,q\in \prop(B),\ p \le q \quad\Longrightarrow\quad \supp p \subseteq \supp q;
\end{equation}
here \(p \le q\) by definition if there is \(p'\in\prop(B)\)
with \(p\oplus p' = q\).

\begin{definition}
  \label{def:Idealsc}
  Let \(\Idealsc(B) \subseteq \Ideals(B)\)
  be the range of the map~\(\supp\).
  We also say that an ideal \emph{is generated by a projection} if it
  is equal to \(\supp p\)
  for some projection \(p\in\Mat_\infty(B)\).
  Let \(\Idealsc^\infty(B) \subseteq \Ideals(B)\)
  be the larger subset of all ideals in~\(B\)
  that are generated by some set of projections in \(\Mat_\infty(B)\).
\end{definition}

The lattice~\(\Ideals(B)\)
is isomorphic to the lattice of open subsets of the topological
space~\(\widehat{B}\)
of irreducible representations of~\(B\)
(see \cite{Dixmier:Cstar-algebres}*{\S3.2}).  Let
\(\widehat{I}\subseteq \widehat{B}\)
denote the open subset of~\(\widehat{B}\)
corresponding to \(I\in\Ideals(B)\).
This space is indeed homeomorphic to the space of irreducible
representations of~\(I\).

\begin{lemma}
  \label{lem:ideal_gen_by_projection}
  Let~\(B\)
  be a \Cstar{}algebra
  and \(I\in\Ideals(B)\).
  Then \(I\in\Idealsc(B)\)
  if and only if \(I\in\Idealsc^\infty(B)\)
  and~\(\widehat{I}\) is quasi-compact.
\end{lemma}

\begin{proof}
  Let \(p\in\Mat_n(B)\)
  for some \(n\in\N\) be idempotent.  Let~\(I\)
  be the ideal generated by~\(p\).
  An irreducible representation~\(\varrho\)
  of~\(B\)
  remains nondegenerate on~\(I\)
  if and only if \(\varrho_n(p)\neq0\),
  if and only if \(\norm{\varrho_n(p)}\ge 1\),
  where~\(\varrho_n\)
  is the corresponding irreducible representation of~\(\Mat_n(B)\).
  So \(\widehat{I}\subseteq \widehat{B}\)
  is the set of all \([\varrho]\)
  with \(\norm{\varrho_n(p)}\ge 1\).
  This subset is quasi-compact by
  \cite{Dixmier:Cstar-algebres}*{Proposition~3.3.7}.
  Conversely, let~\(I\)
  be generated by some set of
  projections~\((p_\alpha)_{\alpha\in S}\)
  and let \(\widehat{I}\subseteq \widehat{B}\)
  be quasi-compact.  For each finite subset \(F\subseteq S\),
  let~\(I_F\)
  be the ideal generated by the projections~\(p_\alpha\)
  for \(\alpha\in F\).
  Then \(\bigcup I_F = I\).
  Since~\(\widehat{I}\)
  is quasi-compact, there is some finite subset \(F\subseteq S\)
  with \(\widehat{I_F}=\widehat{I}\)
  and hence \(I_F=I\).
  Thus~\(I\)
  is generated by the projection \(\bigoplus_{\alpha\in F} p_\alpha\).
\end{proof}

\begin{remark}
  A \Cstar{}algebra
  has the \emph{ideal property} if
  \(\Ideals(B) = \Idealsc^\infty(B)\), that is,
  any ideal is generated by the projections that it contains.
  Then an ideal \(I\idealin B\)
  belongs to~\(\Idealsc(B)\)
  if and only if~\(\widehat{I}\)
  is quasi-compact by Lemma~\ref{lem:ideal_gen_by_projection}.  Any
  \Cstar{}algebra
  of real rank~\(0\)
  has the ideal property by
  \cite{Brown-Pedersen:C*-algebras_of_RR0}*{Corollary~2.8 and
    Proposition~2.9}.  
\end{remark}

\begin{remark}
  If~\(B\)
  is separable and purely infinite, then \(I\idealin B\)
  belongs to~\(\Idealsc(B)\)
  if and only if~\(\widehat{I}\)
  is quasi-compact (see
  \cite{Pasnicu-Rordam:Purely_infinite_rr0}*{Proposition~2.7}).
\end{remark}

\begin{definition}[\cites{Ara-Moreno-Pardo:Nonstable_K_for_Graph_Algs,
    Brown-Pedersen:Non-stable}]
  \label{def:stable_weak_cancellation}
  A \Cstar{}algebra~\(B\)
  has \emph{stable weak cancellation} if projections \(p,q\)
  in~\(\Mat_\infty(B)\)
  with \(\supp p = \supp q\)
  and \([p]=[q]\) in \(\K_0(\supp p)\) are equivalent.
\end{definition}

The equality \([p]=[q]\)
in \(\K_0(\supp p)\)
means that there is a projection \(r\in \Mat_\infty(\supp p)\)
so that \(p\oplus r\)
and~\(q\oplus r\)
are equivalent, briefly, \(p\oplus r \sim q \oplus r\).
So stable weak cancellation holds if and only if \(p\sim q\)
follows if \(p\oplus r \sim q \oplus r\)
for a projection \(r \in \Mat_\infty(\supp p \cap \supp q)\):
this can only happen if \(\supp p = \supp q\).

If~\(B\)
is simple, then either \(p=0\)
or \(\supp p = B\).
Thus stable weak cancellation for simple~\(B\)
means cancellation for all nonzero projections.  It allows that there
is a full (equivalently, nonzero) projection \(0^\full\in\prop(B)\)
representing~\(0\)
in~\(\K_0(B)\).
Then~\(\prop(B)\)
does not have cancellation because
\(0^\full + 0^\full = 0 + 0^\full\)
in~\(\prop(B)\).
Thus stable weak cancellation is strictly weaker than cancellation in
the monoid~\(\prop(B)\).

Graph \Cstar{}algebras have stable weak cancellation by
\cite{Ara-Moreno-Pardo:Nonstable_K_for_Graph_Algs}*{Corollary~7.2}.

\begin{theorem}
  \label{the:prop_for_pi}
  Purely infinite \Cstar{}algebras have stable weak cancellation.
\end{theorem}

\begin{proof}
  This follows from \cite{Rordam:Classification_survey}*{Theorem~4.1.4}.
\end{proof}

Our next goal is to prove the following: if~\(B\)
has stable weak cancellation, then the monoid~\(\prop(B)\)
contains the same information as the diagram of monoids~\(\K_0^+(I)\)
for \(I\in\Idealsc(B)\)
with the canonical maps \(\K_0^+(I) \to \K_0^+(J)\)
for \(I,J\in\Idealsc(B)\)
with \(I\subseteq J\).  We first make the statement more precise.

\begin{definition}
  \label{def:diagrams}
  Let~\((L,\le)\)
  be a partially ordered set.  A \emph{diagram of
    \textup{(}semi\textup{)}groups} over
  it consists of (semi)groups~\(X_W\)
  for \(W\in L\)
  and homomorphisms \(\iota_{W_2,W_1}\colon X_{W_1} \to X_{W_2}\)
  for \(W_1\le W_2\)
  in~\(L\),
  such that \(\iota_{W,W}=\id_{X_W}\)
  for all \(W\in L\)
  and \(\iota_{W_3,W_2}\circ \iota_{W_2,W_1} = \iota_{W_3,W_1}\)
  if \(W_1 \le W_2\le W_3\) in~\(L\).

  Given diagrams \(X=(X_W,\varphi_{W_2,W_1})\)
  and \(Y=(Y_W,\psi_{W_2,W_1})\)
  over~\(L\),
  we let \(\Hom_L(X,Y)\)
  be the (semi)group of natural transformations of diagrams; its
  elements are families of homomorphisms
  \(X_W \to Y_W\)
  for \(W\in L\)
  that intertwine the maps \(\varphi_{W_2,W_1}\)
  and~\(\psi_{W_2,W_1}\) for \(W_1\subseteq W_2\) in~\(L\).

  Let \((L_1,\le)\)
  and \((L_2,\le)\)
  be partially ordered sets and \(\psi\colon L_1\to L_2\)
  an order-preserving map.  Let \((X_W,\varphi_{W_2,W_1})_{W\in L_2}\)
  be a diagram of (semi)groups over~\(L_2\).
  Then \(\psi^*((X_W,\varphi_{W_2,W_1})_{W\in L_2})\)
  is the diagram of (semi)groups over~\(L_1\)
  consisting of the (semi)groups \(X_{\psi(W)}\)
  for \(W\in L_1\)
  and the homomorphisms \(\varphi_{\psi(W_2),\psi(W_1)}\)
  for \(W_1,W_2\in L_1\) with \(W_1\le W_2\).
\end{definition}

\begin{definition}
  \label{def:K0_pi}
  Let~\(B\)
  be a \Cstar{}algebra.  Let \(\K_0^+(B,\Idealsc)\)
  be the diagram over~\(\Idealsc(B)\)
  consisting of the Abelian monoids~\(\K_0^+(I)\)
  for \(I\in\Idealsc(B)\)
  and the homomorphisms \(\K_0^+(I) \to \K_0^+(J)\)
  induced by the inclusion maps \(I \to J\)
  for \(I,J\in\Idealsc(B)\)
  with \(I\subseteq J\).
  We define \(\K_0(B,\Idealsc)\)
  and \(\K_1(B,\Idealsc)\)
  similarly, replacing~\(\K_0^+\)
  by \(\K_0\) and~\(\K_1\).
\end{definition}

We recall the definition of an order unit:

\begin{definition}
  \label{def:order_unit}
  An \emph{order unit} of a commutative semigroup~\(M\)
  (with its natural preorder~\(\leq\))
  is an element~\(u\)
  of~\(M\)
  such that for any element~\(x\)
  of~\(M\),
  there is a positive integer~\(n\)
  with \(x \le n\cdot u\).
  A homomorphism \(\varphi\colon M\to M'\)
  \emph{preserves order units} if~\(\varphi(u)\)
  for an order unit~\(u\) in~\(M\) is an order unit in~\(M'\).
  (If \(\varphi(u)\)
  is an order unit for some order unit~\(u\)
  in~\(M\),
  then this happens for all order units~\(u\) in~\(M\).)
\end{definition}

Let \(A\)
and~\(B\)
be \Cstar{}algebras
and let \(\psi\colon \Idealsc(A) \to \Idealsc(B)\)
be monotone.  Then the pull-back diagram
\(\psi^* \bigl(\K_0^+(B,\Idealsc) \bigr)\)
consists of the monoids \(\K_0^+ \bigl(\psi(I) \bigr)\)
for \(I\in\Idealsc(A)\)
with the canonical maps \(\K_0^+ \bigl(\psi(I_1) \bigr) \to \K_0^+
\bigl(\psi(I_2) \bigr)\)
for \(I_1 \subseteq I_2\).

\begin{theorem}
  \label{the:prop_vs_K0plus}
  Let \(A\)
  and~\(B\)
  be \Cstar{}algebras.
  Let~\(B\)
  have stable weak cancellation.  Then there is a natural bijection
  between monoid homomorphisms \(f\colon \prop(A) \to \prop(B)\)
  and pairs \((\psi,(\varphi_I)_{I\in\Idealsc(A)})\),
  where \(\psi\colon \Idealsc(A) \to \Idealsc(B)\)
  is monotone and
  \[
    \varphi=(\varphi_I)_{I\in\Idealsc(A)}\colon
    \K_0^+(A,\Idealsc) \to \psi^*\K_0^+(B,\Idealsc)
  \]
  is a morphism of monoid diagrams such that each~\(\varphi_I\)
  preserves order units.  The map~\(\psi\) in such a pair is
  automatically additive, that is,
  \(\psi(I_1+I_2) = \psi(I_1)+\psi(I_2)\) for all
  \(I_1,I_2\in\Idealsc(A)\).
\end{theorem}

A monoid isomorphism clearly preserves order units.  Hence
Theorem~\ref{the:prop_vs_K0plus} implies that
\(\prop(A) \cong \prop(B)\) if and only if there are an isomorphism
of partially ordered sets
\(\psi\colon \Idealsc(A) \congto \Idealsc(B)\) and an isomorphism
\(\K_0^+(A,\Idealsc) \cong \psi^* \bigl(\K_0^+(B, \Idealsc) \bigr)\)
as diagrams
of monoids over~\(\Idealsc(A)\).  This is what was meant by saying
that \(\prop(A)\) and the diagram \(\K_0^+(I)_{I\in\Idealsc(A)}\)
contain equivalent information.

We prepare the proof of Theorem~\ref{the:prop_vs_K0plus} by two
lemmas.

\begin{lemma}
  \label{lem:full_order_unit}
  Let~\(B\)
  be a \Cstar{}algebra,
  \(I\in\Idealsc(B)\),
  and \(p\in\prop(I)\).
  Then \(\supp p = I\)
  if and only if~\(p\)
  is an order unit in~\(\prop(I)\).
  When~\(B\)
  has stable weak cancellation, this happens if and only if the image
  of~\(p\) in~\(\K_0^+(I)\) is an order unit.
\end{lemma}

\begin{proof}
  Since \(I\in\Idealsc(B)\),
  there is \(q\in\prop(B)\)
  with \(I = \supp q\).  We fix such a~\(q\) throughout the proof.

  Assume that~\(p\)
  is an order unit in~\(\prop(I)\).
  Then \(q \le n\cdot p\)
  for some \(n\in\N_{>0}\)
  and hence \(I = \supp q \subseteq n\cdot \supp p = \supp p\).
  Then \(I = \supp p\).
  Conversely, let \(\supp p= I\).
  Let \(x\in \prop(I)\).
  We must show that \(x \le n\cdot p\)
  for some \(n\in\N_{>0}\).
  Equivalently, \(x\cdot I^\infty\)
  is a direct summand in \((p\cdot I^\infty)^n\)
  as a Hilbert \(I\)\nb-module.

  Let \(\Hilm \defeq p\cdot I^\infty\)
  as a right Hilbert \(I\)\nb-module.
  Then \(\Comp(\Hilm) \cong p(I\otimes \Comp)p\).
  Let~\(\Hilm^*\)
  be the dual Hilbert \(I,\Comp(\Hilm)\)-bimodule.
  Then \(\Hilm^* \otimes_{\Comp(\Hilm)} \Hilm \cong I\)
  and \(\Hilm \otimes_I \Hilm^* \cong \Comp(\Hilm)\).
  So \(x\cdot I^\infty\)
  is a direct summand in \((p\cdot I^\infty)^n\)
  if and only if \(x\cdot I^\infty \otimes_I \Hilm^*\)
  is a direct summand in
  \((p\cdot I^\infty)^n \otimes_I \Hilm^* = \Comp(\Hilm)^n\)
  as a Hilbert \(\Comp(\Hilm)\)-module.
  The \Cstar{}algebra
  \(\Comp(x\cdot I^\infty\otimes_I \Hilm^*)\)
  is isomorphic to \(\Comp(x\cdot I^\infty)\)
  and hence still unital.  So \(x\cdot I^\infty\otimes_I \Hilm^*\)
  is a proper Hilbert \(\Comp(\Hilm)\)-module.
  By Theorem~\ref{the:prop_CB}, \(x\cdot I^\infty\otimes_I \Hilm^*\)
  is a direct summand in \((p\cdot I^\infty)^n\)
  for some \(n\in\N_{>0}\).
  This finishes the proof that \(\supp p=I\) if and only if~\(p\)
  is an order unit in~\(\prop(I)\).

  Let~\(p\)
  be an order unit in~\(\prop(I)\).
  If \(x\in\K_0^+(I)\),
  then \(x=[\bar{x}]\)
  for some \(\bar{x}\in\prop(I)\).
  By assumption, \(\bar{x} \le n\cdot p\)
  for some \(n\in\N_{>0}\).
  Hence \(x \le n\cdot [p]\).
  Thus \([p]\in \K_0^+(I)\)
  is an order unit in~\(\K_0^+(I)\).
  Conversely, let \([p]\in \K_0^+(I)\)
  be an order unit.  Let \(x\in\prop(I)\).
  Then \(x\oplus q\)
  is full in~\(I\).
  There is \(n\in\N_{>0}\)
  with \([x\oplus q] \le n\cdot [p]\)
  in~\(\K_0^+(I)\),
  that is, \([x\oplus q] \oplus [y] = n\cdot [p]\)
  for some \(y\in \prop(I)\).
  Since~\(B\)
  has stable weak cancellation and both sides are full in~\(I\),
  this implies \(x\oplus q \oplus y \sim n\cdot p\)
  in \(\prop(I)\).
  So \(x \le n\cdot p\).
  Thus~\(p\) is also an order unit in~\(\prop(I)\).
\end{proof}

\begin{definition}
  \label{def:supp_relation}
  For \(p,q\in \prop(B)\), we write \(p\prec q\) if and only if
  \(\supp p \subseteq \supp q\).
\end{definition}

\begin{lemma}
  \label{lem:supp_relation}
  Let \(p,q\in \prop(B)\).
  Then \(p\prec q\)
  if and only if \(p \le n\cdot q\)
  for some \(n\in\N\).  
\end{lemma}

\begin{proof}
  If \(p \le n\cdot q\),
  then \(\supp p \subseteq \supp n\cdot q = \supp q\).
  Conversely, let \(\supp p \subseteq \supp q\).
  By Lemma~\ref{lem:full_order_unit}, \(q\)
  is an order unit for \(\prop(\supp q)\).
  Hence \(p\le n\cdot q\) for some \(n\in\N\).
\end{proof}

\begin{proof}[Proof of Theorem~\textup{\ref{the:prop_vs_K0plus}}]
  Let \(f\colon \prop(A) \to \prop(B)\) be a homomorphism.  The map
  \(\supp\colon \prop(A) \to (\Idealsc(A),+)\) is a surjective
  monoid homomorphism.  Hence~\(\Idealsc(A)\) is isomorphic to the
  quotient of~\(\prop(A)\) by the congruence relation defined by
  \(\supp p = \supp q\).  Lemma~\ref{lem:supp_relation} describes
  this congruence relation and also the partial order defined by
  \(\supp p \subseteq \supp q\) in a purely algebraic way.  So
  \(\supp p \subseteq \supp q\) implies
  \(\supp f(p) \subseteq \supp f(q)\) and similarly with equality
  instead of inclusion.  As a result, \(f\) induces a well defined
  map \(\psi\colon \Idealsc(A) \to \Idealsc(B)\) that is both
  additive and monotone.

  Let \(I\in\Idealsc(A)\).  Then there is \(q\in\prop(A)\) with
  \(I= \supp q\).  Then \(\psi(I) = \supp f(q)\).
  Lemma~\ref{lem:supp_relation} shows that \(p\in\prop(A)\)
  satisfies \(p\in \prop(I)\) if and only if there is \(n\in\N\)
  with \(p \le n\cdot q\).  Hence~\(f\) maps \(\prop(I)\) to
  \(\prop(\psi(I)))\).  This induces a group homomorphism
  \(\K_0(I) \to \K_0(\psi(I))\), which restricts to a monoid
  homomorphism \(\varphi_I\colon \K_0^+(I) \to \K_0^+(\psi(I))\).
  It maps the order unit~\(q\) in \(\prop(I)\) to \(f(q)\), which is
  an order unit in \(\prop(\psi(I))\) by
  Lemma~\ref{lem:full_order_unit} because \(\psi(I) = \supp f(q)\).
  These homomorphisms for \(I\in\Idealsc(A)\) form a morphism of
  diagrams over~\(\Idealsc(A)\) that preserves order units.

  Conversely, let \(\psi\colon \Idealsc(A) \to \Idealsc(B)\) be a
  monotone map and let
  \(\varphi_I\colon \K_0^+(I) \to \K_0^+(\psi(I))\) for
  \(I\in\Idealsc(A)\) be monoid homomorphisms that preserve order
  units and form a morphism of diagrams.  Assume that~\(B\) has
  stable weak cancellation.  Let \(p\in \prop(A)\).  Let
  \(I\defeq \supp p\).  Choose \(q\in \prop(\psi(I))\)
  representing~\(\varphi_I[p]\) in \(\K_0^+(\psi(I))\).
  Since~\(\varphi_I\) preserves order units, \(\supp q = \psi(I)\).
  Since~\(B\) has stable weak cancellation, \(q\) is unique as an
  element of \(\prop(B)\).  So \(f(p) \defeq q\) well-defines a map
  \(f\colon \prop(A)\to \prop(B)\).

  Let \(p_1,p_2\in\prop(A)\).  We claim that
  \(f(p_1+p_2) = f(p_1)+f(p_2)\).  Let \(p_3\defeq p_1+p_2\).  For
  \(j=1,2,3\), define \(I_j \defeq \supp p_j\).  By
  Lemma~\ref{lem:full_order_unit}, \([p_j]\in \K_0^+(I_j)\) is an
  order unit.  Therefore,
  \(\varphi_{I_j}[p_j] \in \K_0^+(\psi(I_j))\) is an order unit.
  Since~\(B\) has stable weak cancellation,
  \(q_j \defeq f(p_j)\in \prop(B)\) is the unique element with
  support ideal~\(\psi(I_j)\) and \(\varphi_{I_j}[p_j] = [q_j]\) in
  \(\K_0^+(\psi(I_j))\).  The relation \(p_1+ p_2 = p_3\) holds in
  \(\K_0^+(I_3)\) when we tacitly view \(p_1,p_2,p_3\) as elements
  of \(\prop(I_3)\).  Since \(\varphi_{I_3}\) is additive, this
  implies the relation \([q_1+q_2] = [q_1]+[q_2] = [q_3]\) in
  \(\K_0^+(\psi(I_3))\).  Since the maps~\(\varphi_I\) preserve
  order units, \([q_1+q_2] = [q_3]\) is an order unit for
  \(\K_0^+(\psi(I_3))\).  Then Lemma~\ref{lem:full_order_unit}
  implies that \(\psi(I_3)\) is the support ideal of \(q_3\) and of
  \(q_1+q_2\).  The support of \(q_1+q_2\) is contained in
  \(\psi(I_1) + \psi(I_2)\).  So
  \(\psi(I_3) \subseteq \psi(I_1) + \psi(I_2)\).  Since~\(\psi\) is
  monotone by assumption, it follows that~\(\psi\) must be additive.
  And stable weak cancellation for~\(B\) shows that
  \(q_1+q_2 = q_3\) holds in \(\prop(B)\).  That is,
  \(f\colon \prop(A) \to \prop(B)\) is a homomorphism.  It is easy
  to see that the two constructions above are inverse to each
  other, giving the asserted bijection.
\end{proof}

\begin{corollary}
  \label{cor:Hom_prop_purely_infinite}
  Let \(A\)
  and~\(B\)
  be \Cstar{}algebras.
  Assume that~\(B\)
  is purely infinite.  There is a natural bijection
  between homomorphisms \(f\colon \prop(A) \to \prop(B)\)
  and pairs \((\psi,(\varphi_I)_{I\in\Idealsc(A)})\)
  with a monotone map \(\psi\colon \Idealsc(A) \to \Idealsc(B)\)
  and a morphism of diagrams of Abelian groups
  \(\varphi_I\colon\K_0(A,\Idealsc) \to \psi^*\K_0(B,\Idealsc)\).
\end{corollary}

\begin{proof}
  Since any \(I\in\Idealsc(B)\)
  is also purely infinite, any class in~\(\K_0(I)\)
  is represented by a projection in~\(I\)
  (see~\cite{Rordam:Classification_survey}*{Theorem~4.1.4}).  Thus
  \(\K_0^+(B,\Idealsc) = \K_0(B,\Idealsc)\),
  which is a diagram of Abelian groups.  Any element of an Abelian
  group is an order unit.  Theorem~\ref{the:prop_for_pi} says
  that~\(B\)
  has stable weak cancellation.  So Theorem~\ref{the:prop_vs_K0plus}
  applies here and gives the remaining statements with~\(\K_0^+\)
  instead of~\(\K_0\).
  If \(I\in\Idealsc(A)\),
  then \(\K_0(I)\)
  is isomorphic to the Grothendieck group of~\(\K_0^+(I)\).
  Hence there is a bijection between homomorphisms
  \(\K_0^+(I) \to \K_0^+(J)\)
  and \(\K_0(I) \to \K_0(J)\)
  for \(I\in\Idealsc(A)\) and \(J\in\Idealsc(B)\).
\end{proof}

\begin{remark}
  \label{rem:Idealsc_infty}
  If \(I\in\Idealsc^\infty(B)\),
  then the set of \(J\in\Idealsc(B)\)
  with \(J\subseteq I\)
  is directed, and the closure of its union is equal to~\(I\).
  If \(I\in\Idealsc(B)\),
  then this directed set has~\(I\)
  as a maximal element.  Therefore, any monoid homomorphism
  \(\Idealsc(A) \to \Idealsc(B)\)
  extends uniquely to a monoid homomorphism
  \(\Idealsc^\infty(A) \to \Idealsc^\infty(B)\)
  that commutes also with infinite suprema.  We also have
  \(\K_0^+(I) = \varinjlim \K_0^+(J)\)
  and \(\K_i(I) = \varinjlim \K_i(J)\)
  for \(i=1,2\),
  where~\(J\)
  runs through the directed set of all \(J\in\Idealsc(B)\)
  with \(J\subseteq I\).
  Thus a diagram morphism
  \(\K_0^+(A,\Idealsc) \to \psi^*\K_0^+(B,\Idealsc)\)
  or \(\K_i(A,\Idealsc) \to \psi^*\K_i(B,\Idealsc)\),
  \(i=0,1\),
  extends uniquely to a diagram morphism
  \(\K_0^+(A,\Idealsc^\infty) \to \psi^*\K_0^+(B,\Idealsc^\infty)\)
  or \(\K_i(A,\Idealsc^\infty) \to \psi^*\K_i(B,\Idealsc^\infty)\).
  Hence we may replace the indexing set \(\Idealsc(A)\)
  by \(\Idealsc^\infty(A)\)
  in Theorem~\ref{the:prop_vs_K0plus} and
  Corollary~\ref{cor:Hom_prop_purely_infinite} provided we add the
  condition that the map on~\(\Idealsc^\infty\) preserves infinite
  suprema.  The condition in
  Theorem~\ref{the:prop_vs_K0plus} that the map
  \(\varphi_I\colon \K_0^+(I) \to \K_0^+(\psi(I))\)
  preserves order units is empty for
  \(I\in\Idealsc^\infty(A) \setminus \Idealsc(A)\)
  because~\(\K_0^+(I)\) has no order unit in this case.
\end{remark}

\section{Some homological algebra with K-theory diagrams}
\label{sec:diagrams}

We first recall well known results on the ideal structure of
graph \Cstar{}algebras
and about their \(\K\)\nb-theory
diagrams \(\K_0(\Cst(G),\Idealsc)\),
\(\K_1(\Cst(G),\Idealsc)\)
and \(\K_0^+(\Cst(G),\Idealsc)\).

\begin{definition}
  A subset \(H\subseteq V\) is \emph{hereditary} if for any edge
  \(e\in E\), \(r(e)\in H\) implies \(s(e) \in H\).  It is
  \emph{saturated} if any \(v\in V_\reg\) with \(s(e) \in H\) for
  all edges~\(e\) with \(r(e)=v\) satisfies \(v\in H\).
\end{definition}

The intersection of any set of hereditary and saturated subsets is
again hereditary and saturated.  The union may fail to be hereditary
and saturated, but then there is a smallest hereditary and saturated
subset containing it.  Thus the hereditary and saturated subsets
of~\(V\) form a complete lattice, which we denote by \(\Ideals(G)\).

\begin{definition}
  For a subset \(V_2\subseteq V\), let
  \(\langle V_2\rangle\subseteq V\) be the smallest hereditary and
  saturated subset containing~\(V_2\).  If \(V_2=\{v\}\), we briefly
  denote this by~\(\langle v\rangle\).
\end{definition}

There is an order-isomorphism between \(\Ideals(G)\) and the lattice
\(\Ideals^\T(\Cst(G))\) of gauge-invariant ideals in~\(\Cst(G)\)
because~\(G\) is row-finite (see
\cite{Bates-Pask-Raeburn-Szymanski:Row_finite}*{Thereom~4.1} or
\cite{Bates-Hong-Raeburn-Szymanski:Ideal_structure}*{Theorem~3.6}).
It maps a hereditary and saturated subset \(H\subseteq V\) to the
ideal generated by the projections~\(p_v\) for all \(v\in H\).
Its inverse maps a gauge-invariant ideal~\(I\) to the set of
\(v\in V\) with \(p_v\in I\).  We write \(\Cst(G)_H\) for the ideal
corresponding to \(H\in\Ideals(G)\).  It is Morita--Rieffel
equivalent to the graph \Cstar{}algebra of the restriction~\(G|_H\)
of~\(G\) to~\(H\).

Write \(q\in\prop(\Cst(G))\) as \(\sum_{v\in V} c_v p_v\) as in
Theorem~\ref{the:prop_graph}.  Then \(\supp q\) is equal to the
ideal generated by the projections~\(p_v\) with \(c_v\neq0\).  An
ideal of this form is automatically gauge-invariant.  The
hereditary, saturated subset of~\(V\) that corresponds to
\(\supp q\) is the smallest one that contains all~\(v\) with
\(c_v\neq0\).  It is called the (hereditary, saturated)
\emph{support} of~\(c\).  The following theorem summarises the
findings:

\begin{theorem}
  \label{the:graph_Idealsc}
  Let~\(G\) be a graph.  An ideal of~\(\Cst(G)\) belongs to
  \(\Idealsc^\infty(\Cst(G))\) if and only if it is gauge-invariant,
  and \(\Idealsc^\infty(\Cst(G))\cong\Ideals(G)\) as complete
  lattices.  This isomorphism maps the subset \(\Idealsc(\Cst(G))\)
  onto the partially ordered set of \emph{finitely generated}
  hereditary, saturated subsets of~\(V\); we denote this
  by~\(\Idealsc(G)\).  The support map
  \(\supp \colon \prop(\Cst(G)) \to \Idealsc(\Cst(G)) \cong
  \Idealsc(G)\) maps the class of the projection \(\sum c_v p_v\) for
  \(c\in \N[V]\) to the hereditary, saturated support of~\(c\).
\end{theorem}

The free Abelian groups \(\Z[W]\) for \(W\in\Idealsc(G)\) form a
diagram of
groups over~\(\Idealsc(G)\), which we denote by \(\Z[V,\Idealsc]\).
Similarly, let \(W_\reg \defeq W\cap V_\reg\) for \(W\in\Idealsc(G)\).
The free Abelian  groups \(\Z[W_\reg]\) for \(W\in\Idealsc(G)\) form
a diagram of
groups over~\(\Idealsc(G)\) denoted by \(\Z[V_\reg,\Idealsc]\).  We
define a morphism of diagrams \(M_G^{\Idealsc}\colon
\Z[V_\reg,\Idealsc] \to \Z[V,\Idealsc]\) to consist of the maps
\[
  M_G^W\colon \Z[W_\reg] \to \Z[W],\qquad
  M_G^W\bigl((c_w)_{w\in W}\bigr) (v) \defeq
  \sum_{\setgiven{e\in E}{\s(e)=v}} c_{\rg(e)}.
\]
These maps form a diagram morphism because if \(W\subseteq V\) is
hereditary, then \(M_G^W = M_G^V|_{\Z[W]}\).  Roughly speaking,
\(M_G^{\Idealsc}\) multiplies with the \emph{adjacency matrix}
of~\(G\), that is, with the matrix whose \(v,w\)-entry is the number
of edges \(v\leftarrow w\) in~\(G\).

The results on the \(\K\)\nb-theory
of graph \Cstar{}algebras provide natural exact sequences
\begin{equation}
  \label{eq:K_graph_exact_sequence_at_W}
  0 \to \K_1(\Cst(G|_W)) \xrightarrow{\kappa|_W}
  \Z[W_\reg] \xrightarrow{\id-M_G^W}
  \Z[W] \xrightarrow{\pi|_W}
  \K_0(\Cst(G|_W)) \to 0
\end{equation}
for all \(W\in\Idealsc^\infty(G)\)
(see
\cite{Raeburn-Szymanski:CK_algs_of_inf_graphs_and_matrices}*{Theorem~3.2}).
These are compatible with the diagram structure over~\(\Idealsc(G)\),
so that~\eqref{eq:K_graph_exact_sequence_at_W} gives an exact sequence
of diagrams
\begin{equation}
  \label{eq:K_graph_exact_sequence}
  0 \to \K_1(\Cst(G),\Idealsc) \xrightarrow{\kappa}
  \Z[V_\reg,\Idealsc] \xrightarrow{\id-M_G^{\Idealsc}}
  \Z[V,\Idealsc] \xrightarrow{\pi}
  \K_0(\Cst(G),\Idealsc) \to 0.
\end{equation}
In addition,
\(\K_0^+(\Cst(G|_W)) = \pi|_W(\N[W]) \subseteq \K_0(\Cst(G|_W))\) by
Theorem~\ref{the:prop_graph}.

Diagrams of Abelian groups over~\(\Idealsc(G)\) form an Abelian category.
This allows us to do homological algebra with them.  We shall need
the extension groups in this category, which we denote by
\(\Ext^n_{\Idealsc(G)}\) for \(n\in\N\).  The case \(n=0\) gives
\(\Ext^0=\Hom\).

\begin{proposition}
  \label{pro:ZV_projective}
  The diagrams \(\Z[V_\reg,\Idealsc]\) and \(\Z[V,\Idealsc]\) are
  projective.  So~\eqref{eq:K_graph_exact_sequence} is the beginning
  of a projective resolution of \(\K_0(\Cst(G),\Idealsc)\).
\end{proposition}

\begin{proof}
  For \(v\in V\),
  let \(\langle v\rangle\subseteq V\)
  be the hereditary and saturated subset generated by~\(v\).
  So \(\langle v\rangle\in\Idealsc(G)\).
  Let \((X_W,\varphi_{W_2,W_1})_{W\in\Idealsc(G)}\)
  be a diagram of groups over~\(\Idealsc(G)\).
  A family \(x_v\in X_{\langle v\rangle}\)
  for \(v\in V\)
  defines a diagram morphism
  \(\xi\colon \Z[V,\Idealsc]\to (X_W)_{W\in\Idealsc(G)}\),
  which consists of the maps
  \[
  \xi_W\colon \Z[W] \to X_W,\qquad
  \sum_{v\in W} c_v \delta_v \mapsto \sum_{v\in W} c_v \varphi_{W,\langle v\rangle}(x_v).
  \]
  Any diagram morphism \(\xi\colon \Z[V,\Idealsc]\to
  (X_W)_{W\in\Idealsc(G)}\) is of this form for a unique family of
  \(x_v\in X_{\langle v\rangle}\) for \(v\in V\), namely, \(x_v
  \defeq \xi_{\langle v\rangle}(\delta_v)\).  So
  \[
  \Hom_{\Idealsc(G)} \bigl(\Z[V,\Idealsc], (X_W)\bigr) \cong
  \prod_{v\in W} X_{\langle v\rangle}.
  \]
  This functor is exact.  That is, \(\Z[V,\Idealsc]\) is projective.
  A similar argument works for \(\Z[V_\reg,\Idealsc]\) or, more
  generally, for the diagram \(W\mapsto \Z[W\cap X,\Idealsc]\)
  for any subset~\(X\) of~\(V\).
\end{proof}

The resolution in~\eqref{eq:K_graph_exact_sequence} allows us to compute
\(\Ext^n_{\Idealsc(G)}(\K_0(\Cst(G),\Idealsc),Y)\)
for any diagram of groups~\(Y\)
over~\(\Idealsc(G)\)
and \(n=0,1,2\) as the cohomology groups of the cochain complex
\begin{multline}
  \label{eq:compute_Ext_Idealsc}
  0 \to 
  \Hom_{\Idealsc(G)}(\Z[V,\Idealsc],Y) \xrightarrow{(\id-M_G^{\Idealsc})^*}
  \Hom_{\Idealsc(G)}(\Z[V_\reg,\Idealsc],Y) \\ \xrightarrow{\kappa^*}
  \Hom_{\Idealsc(G)}(\K_1(\Cst(G),\Idealsc),Y) \to 0.
\end{multline}
This works for \(n=0,1,2\), regardless whether or not
\(\K_1(\Cst(G),\Idealsc)\) is projective.

We now compare the resolution above to the constructions
in~\cite{Bentmann-Meyer:More_general}.  That article only considers
graph \Cstar{}algebras with finitely many ideals and, without
loss of generality, assumes all graphs to be regular.  While each
ideal in \(\Cst(G)\) is generated by a projection in that case, only
some ideals are used in~\cite{Bentmann-Meyer:More_general}.  Call an
ideal \(I\idealin B\) \emph{irreducible} if
\(I= \overline{\sum_{\alpha\in S} J_\alpha}\) for a family of
ideals~\(J_\alpha\) implies \(J_\alpha=I\) for some \(\alpha\in S\).
The irreducible ideals form a partially ordered set.  The invariant
that is used in~\cite{Bentmann-Meyer:More_general} is the diagram
consisting of the \(\Z/2\)\nb-graded Abelian groups \(\K_*(I)\) for
all \emph{irreducible} ideals \(I\idealin B\).  If~\(B\) has only
finitely many ideals, then for each primitive ideal
\(\mathfrak{p}\idealin B\) there is a minimal open subset of the
primitive ideal space that contains~\(\mathfrak{p}\).  The ideals
corresponding to these minimal open subsets are exactly the
irreducible ideals.

\begin{remark}
  The Cantor set is the primitive ideal space of a commutative
  AF-algebra and hence of a graph \Cstar{}algebra.  All open
  subsets of the Cantor set are reducible.  So restricting to
  irreducible ideals is not reasonable for \Cstar{}algebras with
  infinitely many ideals, not even if they are graph
  \Cstar{}algebras.
\end{remark}

For a graph \Cstar{}algebra with finitely many ideals, let
\(\Idealss(G)\subseteq \Idealsc(G) = \Ideals(G)\) be the subset
corresponding to the irreducible ideals in~\(\Cst(G)\).  We may
restrict the exact sequence of
diagrams~\eqref{eq:K_graph_exact_sequence} to the subset
\(\Idealss(G)\subseteq \Idealsc(G)\).  The restricted diagrams
\(\Z[V,\Idealss]\) and \(\K_1(\Cst(G),\Idealss)\) are shown
in~\cite{Bentmann-Meyer:More_general} to be projective objects in
the Abelian category of diagrams of Abelian groups
over~\(\Idealss(G)\).  Hence the groups
\(\Ext^n_{\Idealss(G)}\bigl(\K_0(\Cst(G),\Idealss),\psi^*\K_1(B,\Idealss)\bigr)\)
in this Abelian category of diagrams vanish for \(n\ge3\) and are
computed exactly as in~\eqref{eq:compute_Ext_Idealsc}
with~\(\Idealss\) instead of~\(\Idealsc\) if \(n=0,1,2\).

The invariant over~\(\Idealss\)
clearly contains less information than the invariant
over~\(\Idealsc\).
This does not mean, however, that lifting a map on the invariant
over~\(\Idealsc\)
is easier than lifting a map on the invariant over~\(\Idealss\): a
given extension to~\(\Idealsc\)
may put constraints on the desired \Cstar{}correspondence.

\begin{remark}
  \label{rem:irreducible_vs_singly_generated}
  Let~\(G\) be a regular graph with only finitely many hereditary
  saturated subsets.  Then any irreducible hereditary saturated
  subset is generated by a single vertex \(v\in V\).  The converse
  does not hold, however, because the union of two hereditary and
  saturated subsets of~\(V\) need not remain hereditary.  For
  instance, let~\(G\) be the graph
  \[
    \begin{tikzpicture}[node distance=1cm]
      \node (3) {3};
      \node (1) [left of=3] {1};
      \node (2) [right of=3] {2.};
      \path (1) edge [->,loop above] (2) edge[->] (3);
      \path (2) edge [->,loop above] (2) edge[->] (3);
    \end{tikzpicture}
  \]
  The hereditary and saturated subsets for~\(G\) are \(\emptyset\),
  \(\{1\}\), \(\{2\}\), and \(\{1,2,3\}\).  The last one is singly
  generated by~\(3\), but not irreducible.
\end{remark}

\begin{remark}
  \label{rem:invariant_infinite_ideal_lattice}
  The machinery of homological algebra in triangulated categories does
  not apply to the invariant \(\K_*(B,\Idealsc)\),
  not even if the ideal lattice of~\(B\)
  is finite.  This remark explains why.  First we make our invariant
  more precise.  Let~\(X\)
  be a locally quasi-compact topological space and let \((L,\le)\)
  be the partially ordered set of quasi-compact open subsets of~\(X\).
  Let~\(B\)
  be a \Cstar{}algebra
  over~\(X\).
  If \(\Prim(B) \cong X\)
  and~\(B\)
  has the ideal property, then \(L=\Idealsc(B)\).
  Diagrams of \(\Z/2\)\nb-graded
  countable Abelian groups over~\((L,\le)\)
  form a stable Abelian category, and mapping~\(B\)
  to the diagram \(\K_*(B,L) \defeq \bigl(\K_*(I)\bigr)_{I\in L}\)
  is a stable homological functor into this category.  This invariant
  is, however, not ``universal,'' so that the homological algebra
  machinery developed in \cites{Meyer-Nest:Homology_in_KK,
    Meyer:Homology_in_KK_II, Meyer-Nest:Bootstrap,
    Meyer-Nest:Filtrated_K} does not apply directly to it.  Let
  \(I_1,I_2\in L\)
  and assume \(I_1 \cap I_2 \in L\).
  We have \(I_1+I_2\in L\)
  anyway.  The boundary map~\(\partial\)
  in the Mayer--Vietoris sequence
  \[
    \dotsb \to
    \K_*(I_1 \cap I_2) \to 
    \K_*(I_1) \oplus \K_*(I_2) \to 
    \K_*(I_1+I_2) \xrightarrow{\partial}
    \K_{*-1}(I_1 \cap I_2) \to \dotsb
  \]
  is a natural transformation from \(\K_*(I_1+I_2)\)
  to \(\K_{*-1}(I_1\cap I_2)\),
  which is not yet encoded in our Abelian category of diagrams.  Hence
  our invariant cannot be universal.
\end{remark}

Despite Remark~\ref{rem:invariant_infinite_ideal_lattice}, we shall
work with the invariant \(\K_*(B,\Idealsc)\) and use homological
algebra in the category of diagrams over~\(\Idealsc(G)\).  This
works quite well to describe correspondences from graph
\Cstar{}algebras to~\(B\).  In fact, we shall use~\(\prop(B)\)
instead of \(\K_0^+(B,\Idealsc)\) most of the time.  These
invariants are equivalent by Theorem~\ref{the:prop_vs_K0plus}
if~\(B\) has stable weak cancellation.  We do not need to assume
stable weak cancellation when we work directly with~\(\prop(B)\).

\section{Adding \texorpdfstring{$\K_1$}{K₁}}
\label{sec:add_K1}

Theorem~\ref{the:lift_prop_CB} only uses the invariant~\(\prop(B)\),
which is not fine enough for classification.  Now we add information
about~\(\K_1\).
We develop a criterion when a pair of maps
\(\prop(\Cst(G))\to \prop(B)\)
and \(\K_1(\Cst(G),\Idealsc) \to \psi^*\K_1(B,\Idealsc)\)
is induced by a proper \(\Cst(G),B\)-correspondence.
It uses a variant of the obstruction class
in~\cite{Bentmann-Meyer:More_general}.

Let~\(\Hilm\)
be a proper \(\Cst(G),B\)-correspondence.
Describe~\(\Hilm\)
through proper Hilbert \(B\)\nb-modules~\(\Hilm_v\)
for \(v\in V\)
and unitaries
\(U_v\colon \bigoplus_{e\in E^v} \Hilm_{\s(e)} \congto \Hilm_v\)
for \(v\in V_\reg\)
as in Theorem~\ref{the:lift_prop_CB}.  We are going to
describe the homomorphisms
\[
\prop(\Cst(G))\to \prop(B),\quad
\Ideals(\Cst(G)) \to \Idealsc(B),\quad
\K_1(\Cst(G),\Idealsc) \to \psi^*\K_1(B,\Idealsc)
\]
induced by~\(\Hilm\).
By Theorem~\ref{the:prop_graph},
the homomorphism \(\Hilm_*\colon \prop(\Cst(G))\to \prop(B)\),
\([\Hilm[M]]\mapsto [\Hilm[M] \otimes_{\Cst(G)} \Hilm]\),
is determined by its values
\[
\Hilm_*[p_v]
= \Hilm_*[p_v\cdot \Cst(G)]
= [p_v\cdot \Hilm]
= [\Hilm_v].
\]
The homomorphism \(\Hilm_*\colon \prop(\Cst(G))\to \prop(B)\)
induces a map
\[
\psi\colon \Idealsc(G) \cong \Idealsc(\Cst(G)) \to\Idealsc(B)
\]
as in the proof of Theorem~\ref{the:prop_vs_K0plus}; this part of
the proof works without stable weak cancellation.  Namely, let
\(W\in\Idealsc(G)\)
be a finitely generated, hereditary, saturated subset of~\(V\).
Let \(F\subseteq W\)
be any finite generating set for~\(W\).
Then the projection \(\bigoplus_{v\in F} p_v\)
generates the ideal~\(\Cst(G)_W\) corresponding to~\(W\).  Hence
\[
\psi(W)\defeq \supp \varphi_0\Bigl(\sum_{v\in F} {}[p_v]\Bigr)
= \bigvee_{v\in F} {}\supp \varphi_0[p_v]
\]
This ideal is equal to \(\bigvee_{v\in W} {}\supp \varphi_0[p_v]\)
and does not depend on the choice of~\(F\).

By definition, \(\psi(W)\)
is the ideal in~\(B\) that is generated by the inner product on the
proper Hilbert \(B\)\nb-module
\(\bigoplus_{v\in F} p_v\cdot \Hilm = \bigoplus_{v\in F} \Hilm_v\).
Therefore, \(\Cst(G)_W\cdot \Hilm\)
is a full Hilbert \(\psi(W)\)\nb-module.
As such, it is nondegenerate as a right \(\psi(W)\)\nb-module, so that
\begin{equation}
  \label{eq:over_Idealsc}
  \Cst(G)_W \cdot \Hilm \subseteq \Hilm\cdot \psi(W).
\end{equation}
We briefly call a \(\Cst(G),B\)-correspondence~\(\Hilm\)
satisfying~\eqref{eq:over_Idealsc}
\emph{\(\Idealsc(G)\)\nb-covariant}.

Equation~\eqref{eq:over_Idealsc} implies that~\(\Hilm\)
restricts to a proper \(\Cst(G)_W,\psi(W)\)-correspondence
and thus induces a grading-preserving group homomorphism
\[
\Hilm_*|_W\colon\K_*(\Cst(G)_W) \to \K_*(\psi(W)).
\]
These homomorphisms for \(W\in \Idealsc(G)\) form a morphism of diagrams
\[
\Hilm_*\colon\K_*(\Cst(G),\Idealsc) \to \psi^*\K_*(B,\Idealsc).
\]
We shall not use this on~\(\K_0\)
in the following because it contains less information than the
homomorphism \(\Hilm_*\colon \prop(\Cst(G)) \to \prop(B)\)
(compare Theorem~\ref{the:prop_vs_K0plus}).  We now describe the map
\(\Hilm_*|_W\colon\K_1(\Cst(G),\Idealsc) \to \psi^*\K_1(B,\Idealsc)\).
Since the map~\(\kappa|_W\)
in~\eqref{eq:K_graph_exact_sequence_at_W} is injective, we may
identify \(\K_1(\Cst(G)_W) \cong \K_1(\Cst(G|_W))\)
with the kernel of the map \(\id-M_G^W\colon \Z[W_\reg] \to \Z[W]\).
Let \(x = \sum c_v \delta_v\in \Z[W_\reg]\)
belong to this kernel.  Write
\(x = \sum c_v^+\delta_v - \sum c_v^- \delta_v\)
with \(c_v^\pm\ge0\),
\(c_v = c_v^+ - c_v^-\)
and \(c_v^+=0\)
or \(c_v^-=0\)
for all \(v\in W_\reg\).
Set \(c_v=0\)
for \(v\in V_\reg\setminus W\).
The equation \((\id-M_G^W)(x)=0\) says that
\begin{equation}
  \label{eq:kernel_condition}
  \sum_{v\in V_\reg} c_v^+ \delta_v + \sum_{e\in E} c_{\rg(e)}^- \delta_{\s(e)}
  =
  \sum_{e\in E} c_{\rg(e)}^+ \delta_{\s(e)} + \sum_{v\in V_\reg} c_v^- \delta_v.
\end{equation}
Equivalently, the domains and codomains of the partial isometry
\[
u_x\defeq \bigoplus_{v\in W_\reg} u_v^{\oplus c_v^+}\oplus
\bigoplus_{v\in W_\reg} (u_v^*)^{\oplus c_v^-}
\]
are equal up to a permutation of the summands.  Let~\(u'_x\) be the
product of~\(u_x\) with the permutation matrix that makes it a
unitary element in the corner of~\(\Mat_\infty(\Cst(G)_W)\)
associated to the projection
\[
  p_x\defeq \bigoplus_{v\in V_\reg} p_v^{\oplus c_v^+}
  + \bigoplus_{e\in E} p_{\s(e)}^{\oplus c_{\rg(e)}^-}.
\]
Then~\(u'_x\) has a class
in~\(\K_1(\Cst(G)_W)\).  The exact
sequence~\eqref{eq:K_graph_exact_sequence} is proved by showing that
any class in~\(\K_1(\Cst(G)_W)\) is of this form for a unique
\(x\in \ker(\id - M_G^W)\).  Let \(P_v\in\Comp(\Hilm)\) be the
projection onto~\(\Hilm_v\) and let~\(U_v\) be as in
Theorem~\ref{the:lift_prop_CB}.  Define
\[
P_x\defeq \bigoplus_{v\in V_\reg} P_v^{\oplus c_v^+}
\oplus \bigoplus_{e\in E} P_{\s(e)}^{\oplus c_{\rg(e)}^-},
\]
and define a unitary~\(U'_x\)
in~\(P_x \Comp(\Hilm^\infty) P_x\)
in the same way as~\(u'_x\).
This is how \(u'_x \in p_x \Mat_\infty(\Cst(G)) p_x\)
acts on~\(P_x \Hilm^\infty\),
and the latter is a Hilbert \(\psi(W)\)-module.
The map \(\Hilm_*|_W\colon \K_1(\Cst(G)_W) \to \K_1(\psi(W))\)
maps~\(x\) to~\([U'_x]\).

Now let \(\varphi_0\colon \prop(\Cst(G))\to \prop(B)\)
be a monoid homomorphism.  Let
\[
\psi\colon \Idealsc(G) \cong \Idealsc(\Cst(G)) \to \Idealsc(B)
\]
be induced by~\(\varphi_0\).
And let
\(\varphi_1\colon \K_1(\Cst(G),\Idealsc) \to \psi^*\K_1(B,\Idealsc)\)
be a morphism of diagrams over~\(\Idealsc(G)\).
When is there a proper \(\Cst(G),B\)-\alb{}correspondence~\(\Hilm\)
that induces the maps \(\varphi_0\)
and~\(\varphi_1\)?
Our criterion needs a variant of the obstruction class
in~\cite{Bentmann-Meyer:More_general}.  Our variant lives in the
\(\Ext^2\)-group
\[
\Eext^2\bigl(G,\psi^*\K_1(B,\Idealsc)\bigr) \defeq
\Ext^2_{\Idealsc(G)}\bigl(\K_0(\Cst(G),\Idealsc),\psi^*\K_1(B,\Idealsc)\bigr).
\]
The extension group in question is described
in~\eqref{eq:compute_Ext_Idealsc} as the cokernel of the map
\[
\kappa^*\colon
\Hom_{\Idealsc(G)}\bigl(\Z[V_\reg,\Idealsc],\psi^*\K_1(B,\Idealsc)\bigr) \to
\Hom_{\Idealsc(G)}\bigl(\K_1(\Cst(G),\Idealsc),\psi^*\K_1(B,\Idealsc)\bigr)
\]
induced by the canonical injective diagram morphism
\(\kappa\colon \K_1(\Cst(G),\Idealsc) \to \Z[V_\reg,\Idealsc]\)
in~\eqref{eq:K_graph_exact_sequence}.  So elements in
\(\Eext^2\bigl(G,\psi^*\K_1(B,\Idealsc)\bigr)\)
are represented by diagram morphisms from \(\K_1(\Cst(G),\Idealsc)\)
to \(\psi^*\K_1(B,\Idealsc)\),
and such a diagram morphism represents~\(0\)
if and only if it factors through~\(\kappa\).

\begin{definition}
  \label{def:obstruction_class}
  Let \(\varphi_0\colon \prop(\Cst(G)) \to \prop(B)\) be a
  homomorphism.  It induces a map
  \(\psi\colon \Idealsc(G) \cong \Idealsc(\Cst(G)) \to
  \Idealsc(B)\).  Let
  \[
  \varphi_1\colon \K_1(\Cst(G),\Idealsc) \to \psi^*\K_1(B,\Idealsc)
  \]
  be a morphism of diagrams.  By Theorem~\ref{the:lift_prop_CB}, there
  is a correspondence~\(\Hilm\)
  that lifts~\(\varphi_0\).
  The \emph{obstruction class}
  \(\Delta(\varphi_0,\varphi_1) \in
  \Eext^2_{\Idealsc(G)}\bigl(G,\psi^*\K_1(B,\Idealsc)\bigr)\)
  of \((\varphi_0,\varphi_1)\)
  is the image of the diagram morphism
  \(\varphi_1 - \Hilm_* \colon \K_1(\Cst(G),\Idealsc) \to
  \psi^*\K_1(B,\Idealsc)\).
\end{definition}

As the name indicates, the obstruction class obstructs the existence
of a proper correspondence that lifts both maps \(\varphi_0\)
and~\(\varphi_1\):

\begin{proposition}
  \label{pro:obstruction_class_well-defined}
  The obstruction class does not depend on the choice of the proper
  correspondence~\(\Hilm\)
  lifting~\(\varphi_0\).
  It vanishes if there is a proper correspondence~\(\Hilm\)
  that induces the given maps \(\varphi_0\)
  and~\(\varphi_1\) on \(\prop\) and~\(\K_1\).
\end{proposition}

\begin{proof}
  By definition, \(\Delta(\varphi_0,\varphi_1)=0\)
  if \(\Hilm\)
  induces \(\varphi_0\)
  and~\(\varphi_1\)
  on \(\prop\)
  and~\(\K_1\).
  So the second statement follows if the obstruction class does not
  depend on the choice of~\(\Hilm\).
  This is what we are going to prove.

  Let~\(\Hilm'\)
  be another proper correspondence lifting~\(\varphi_0\).
  This involves isomorphic Hilbert \(B\)\nb-modules~\(\Hilm_v\)
  because \([\Hilm_v] = \varphi_0[p_v]\).
  So~\(\Hilm'\)
  is isomorphic to the proper correspondence for a family
  \(((\Hilm_v),(U_v'))\)
  with the same proper Hilbert \(B\)\nb-modules~\(\Hilm_v\).
  By Lemma~\ref{lem:difference_unitaries}, the unitaries~\(U_v'\)
  are of the form \(U'_v = \Upsilon_v U_v\)
  with \(\Upsilon_v \in \U(\Comp(\Hilm_v))\)
  for \(v\in V_\reg\).
  Each~\(\Upsilon_v\)
  has a class in \(\K_1(\Comp(\Hilm_v))\).
  The \Cstar{}algebra~\(\Comp(\Hilm_v)\)
  is Morita--Rieffel equivalent through~\(\Hilm_v\)
  to the ideal in~\(B\)
  generated by the inner products of elements in~\(\Hilm_v\).
  This ideal is equal to \(\psi(\langle v\rangle)\idealin B\),
  where \(\langle v\rangle\subseteq V\)
  is the hereditary and saturated subset generated by~\(v\).
  Thus~\(\Hilm'\) yields a diagram morphism
  \begin{equation}
    \label{eq:difference_class_as_diagram_morphism}
    [\Upsilon]\colon \Z[V_\reg,\Idealsc] \to \psi^*(\K_1(B,\Idealsc)),\qquad
    \delta_v\mapsto [\Upsilon_v],
  \end{equation}
  compare the proof of Proposition~\ref{pro:ZV_projective}.  We claim
  that
  \begin{equation}
    \label{eq:difference_class_boundary}
    \Hilm_* - \Hilm'_*  = [\Upsilon]\circ \kappa\qquad
    \text{as maps }\K_1(\Cst(G),\Idealsc) \to \psi^*\K_1(B,\Idealsc)
  \end{equation}
  with the map~\(\kappa\) in~\eqref{eq:K_graph_exact_sequence}.  So
  \(\Hilm\) and \(\Hilm'\) have the same image in
  \(\Eext^2\bigl(G,\psi^*\K_1(B,\Idealsc)\bigr)\), as desired.  To
  prove~\eqref{eq:difference_class_boundary}, choose a hereditary
  and saturated subset \(W\subseteq V\) and
  \(x\in \K_1(\Cst(G)_W)\).  Write
  \(\kappa(x) = \sum_{v\in W_\reg} {}(c_v^+-c_v^-) \delta_v\) with
  \(c_v^\pm\ge0\).  Let \(v\in W_\reg\).  The Hilbert
  \(B\)\nb-modules
  \(\Hilm[F]_v \defeq \bigoplus_{v\in V} \Hilm_v^{\oplus
    c_v^+}\oplus \bigoplus_{e\in E^v} \Hilm_{\s(e)}^{\oplus c_v^-}\)
  and
  \(\bigoplus_{v\in V} \Hilm_v^{\oplus c_v^-}\oplus \bigoplus_{e\in
    E^v} \Hilm_{\s(e)}^{\oplus c_v^+}\) are isomorphic by a
  permutation~\(\Sigma\) of the summands because
  \((\id- M_G^W)\kappa(x)=0\).  Since their inner products take
  values in \(\psi(W)\idealin B\), they are Hilbert
  \(\psi(W)\)\nb-modules.  The unitaries
  \begin{align*}
    U_x&\defeq \bigoplus_{v\in V_\reg} U_v^{\oplus c_v^+}
         \oplus \bigoplus_{v\in V_\reg} (U_v^*)^{\oplus c_v^-},\\
    U'_x&\defeq \bigoplus_{v\in V_\reg} (\Upsilon_v U_v)^{\oplus c_v^+}
          \oplus \bigoplus_{v\in V_\reg} (U_v^* \Upsilon_v^*)^{\oplus c_v^-}
  \end{align*}  
  are isomorphisms between these Hilbert \(\psi(W)\)\nb-modules as
  well.  So \(\Sigma U_x,\Sigma U_x' \in \U(\Hilm[F]_v)\).  By
  definition, \(\Hilm_*(x)\) and \(\Hilm'_*(x)\) are the classes in
  \(\K_1(\psi(W))\) defined by the unitaries \(\Sigma U_x\)
  and~\(\Sigma U_x'\).  They differ by the class of
  \(\bigoplus_{v\in V_\reg} \Upsilon_v^{\oplus c_v^+}\oplus
  \bigoplus_{v\in V_\reg} (\Upsilon_v^*)^{\oplus c_v^-}\).  This is
  \([\Upsilon](\kappa(x))\) because
  \([\Upsilon_v^*] = - [\Upsilon_v]\) in \(\K_1(\psi(W))\).
\end{proof}

The vanishing of the obstruction class is necessary for the existence
of a proper correspondence~\(\Hilm\)
that lifts \((\varphi_0,\varphi_1)\).
We are going to show that it is also sufficient under an extra
assumption on~\(B\).

\begin{definition}
  \label{def:K1_surjective_injective}
  A unital \Cstar{}algebra~\(B\)
  is \emph{\(\K_1\)\nb-surjective}
  or \emph{\(\K_1\)\nb-injective},
  respectively, if the map \(\pi_0(\U(B)) \to\K_1(B)\)
  is surjective or injective, respectively.
\end{definition}

\begin{proposition}
  \label{pro:lift_to_K0_K1}
  Let~\(G\)
  be a graph, \(B\)
  a \Cstar{}algebra,
  and let \(\varphi_0\colon \prop(\Cst(G)) \to \prop(B)\)
  and
  \(\varphi_1\colon \K_1(\Cst(G),\Idealsc) \to
  \psi^*\K_1(B,\Idealsc)\)
  be given, where \(\psi\colon \Idealsc(G) \to \Idealsc(B)\)
  is induced by~\(\varphi_0\).
  For \(v\in V\),
  let \(\Hilm_v\)
  be proper Hilbert \(B\)\nb-modules
  with \([\Hilm_v] = \varphi_0[p_v]\).
  Assume that the unital \Cstar{}algebras~\(\Comp(\Hilm_v)\)
  for \(v\in V\)
  are \(\K_1\)\nb-surjective.
  There is a proper \(\Cst(G),B\)-correspondence~\(\Hilm\)
  that induces the given maps \(\varphi_0\)
  and~\(\varphi_1\) if and only if \(\Delta(\varphi_0,\varphi_1)=0\).
\end{proposition}

\begin{proof}
  One direction is proven already in
  Proposition~\ref{pro:obstruction_class_well-defined}, without
  assumption on~\(B\).  Here we prove the converse.  So assume that
  the obstruction class vanishes.  We shall modify a proper
  \(\Cst(G),B\)-\alb{}correspondence~\(\Hilm\) that
  lifts~\(\varphi_0\) so that it also lifts~\(\varphi_1\).  The
  vanishing of the obstruction class means that
  \[
  \varphi_1 - \Hilm_* = \alpha\circ\kappa\colon
  \K_1(\Cst(G),\Idealsc) \to \psi^*\K_1(B,\Idealsc)
  \]
  for a diagram morphism
  \(\alpha\colon\Z[V_\reg,\Idealsc] \to \psi^*\K_1(B,\Idealsc)\).
  Describe~\(\Hilm\)
  by a family \(((\Hilm_v),(U_v))\)
  as above.  Let \(v\in V_\reg\)
  and let \(\langle v\rangle\subseteq V\)
  be the hereditary and saturated subset generated by~\(v\).
  Since \(\psi(\langle v\rangle)\idealin B\)
  is Morita--Rieffel equivalent to~\(\Comp(\Hilm_v)\)
  and~\(\Comp(\Hilm_v)\)
  is \(\K_1\)\nb-surjective,
  \(\alpha(\delta_v) \in \K_1(\psi(\langle v\rangle)) \cong \K_1(\Comp(\Hilm_v))\)
  is represented by some \(\Upsilon_v\in \U(\Comp(\Hilm_v))\).

  Define \(U_v'\defeq \Upsilon_v\circ U_v\).  The family
  \(((\Hilm_v),(U_v'))\) gives another proper
  \(\Cst(G),B\)-\alb{}correspondence~\(\Hilm'\) lifting the same
  map~\(\varphi_0\).  Let~\([\Upsilon_v]\) denote the class
  of~\(\Upsilon_v\) in
  \(\K_1(\psi(\langle v\rangle)) \cong \K_1(\Comp(\Hilm_v))\).  As
  above, this defines a morphism of diagrams
  \(\Z[V_\reg,\Idealsc] \to \psi^* \K_1(B,\Idealsc)\).  The proof of
  Proposition~\ref{pro:obstruction_class_well-defined} shows that
  \(\Hilm'_* - \Hilm_*\) is equal to
  \([\Upsilon_v]\circ\kappa = \alpha\circ\kappa = \varphi_1 -
  \Hilm_*\).  As a result, \(\Hilm'_* = \varphi_1\).
\end{proof}

Now assume that the target~\(B\) is a graph \Cstar{}algebra as
well.  The following theorem allows us to apply all of the previous
results to study \(A,B\)-correspondences:

\begin{theorem}
  \label{the:graph_K1-bijective}
  Let \(G_2 = (r,s\colon E_2 \rightrightarrows V_2)\) be a
  regular directed graph and let~\(B\) be its graph
  \Cstar{}algebra.  Let \(p\in\prop(B)\).  Then~\(B\) has stable
  weak cancellation and the unital \Cstar{}algebra
  \(p (B\otimes\Comp) p\) is \(\K_1\)\nb-bijective.
\end{theorem}

\begin{proof}
  Graph \Cstar{}algebras have stable weak cancellation by
  \cite{Ara-Moreno-Pardo:Nonstable_K_for_Graph_Algs}*{Corollary~7.2}.
  By Theorem~\ref{the:prop_graph}, any \(p\in\prop(B)\) is
  equivalent to \(\sum_{v\in V_2} c_v [p_v]\) for a finitely
  supported function \((c_v)_{v\in V_2} \in \N[V]\).

  For a finite subset \(V'_2\subseteq V_2\), let~\(B'\) be the
  \Cstar{}subalgebra of~\(B\) generated by
  \(\setgiven{t_e}{r(e)\in V'_2} \cup \setgiven{p_v}{v\in V'_2}\).
  These form a directed set of \Cstar{}subalgebras of~\(B\), and
  their union is dense in~\(B\).  Therefore, \(p (B\otimes\Comp) p\)
  is the inductive limit of the \Cstar{}algebras
  \(p (B'\otimes\Comp) p\).  Both~\(\K_1\) and the set of connected
  components of the unitary group are continuous for inductive
  limits.  Therefore, it is enough to prove that
  \(p (B'\otimes\Comp) p\) is \(\K_1\)\nb-bijective for all finite
  subsets \(V'_2\subseteq V_2\).  We may assume without loss of
  generality that~\(V'_2\) contains \(\supp (c_v)\), so that
  \(p\in B'\otimes \Comp \subseteq B\otimes \Comp\).  The
  \Cstar{}algebra~\(B'\) may be identified with the
  \Cstar{}algebra of a finite graph (compare
  \cite{Raeburn-Szymanski:CK_algs_of_inf_graphs_and_matrices}*{Lemma~1.2}).
  Here we may take the graph with vertex set
  \(V'_2 \cup s(r^{-1}(V'_2))\) and edge set~\(r^{-1}(V'_2)\).
  Therefore, it suffices to prove that \(p (B\otimes\Comp) p\) is
  \(\K_1\)\nb-bijective when~\(B\) is the graph \Cstar{}algebra
  of a finite graph.  We assume this from now on.

  The advantage in the finite case is that there are only finitely
  many gauge-invariant ideals.  Therefore, there is an increasing
  chain of gauge-invariant ideals
  \(0 = I_0 \subset I_1 \subset I_2 \subset \dotsb \subset I_\ell =
  B\) so that each subquotient \(I_{k+1}/I_k\) for
  \(k=0,\dotsc,\ell-1\) has no gauge-invariant ideals.  It is known
  that these subquotients \(I_{k+1}/I_k\) are again Morita
  equivalent to graph \Cstar{}algebras, so that they have stable
  weak cancellation.  In addition, they are either AF or of the form
  \(\Cont(\T,\Mat_n)\) for some \(n\in\N_{\ge1}\) or purely infinite
  simple.  In the first two cases,
  \(p\cdot (I_{k+1}/I_k \otimes \Comp) p\) has stable rank~\(1\)
  (see
  \cite{Raeburn-Szymanski:CK_algs_of_inf_graphs_and_matrices}*{Proposition~5.4}).
  In all three cases, \(p\cdot (I_{k+1}/I_k \otimes \Comp)\cdot p\)
  is extremally rich and has weak cancellation (see
  \cite{Brown-Pedersen:Non-stable}).  Brown and Pedersen show
  in~\cite{Brown-Pedersen:Non-stable} that extremally rich
  \Cstar{}algebras with weak cancellation have the following
  properties:
  \begin{itemize}
  \item \(\K_1\)\nb-surjectivity by
    \cite{Brown-Pedersen:Non-stable}*{Theorem 4.4};
  \item good index theory by
    \cite{Brown-Pedersen:Non-stable}*{Theorem~5.1};
  \item \(\K_1\)\nb-injectivity and weak \(\K_0\)\nb-surjectivity by
    \cite{Brown-Pedersen:Non-stable}*{Theorem 6.7}.
  \end{itemize}
  It is shown in~\cite{Brown-Pedersen:Non-stable} that the class of
  \Cstar{}algebras with these four properties --~weak
  \(\K_0\)\nb-surjectivity, \(\K_1\)\nb-surjectivity,
  \(\K_1\)\nb-injectivity, and good index theory~-- is closed under
  extensions (see \cite{Brown-Pedersen:Non-stable}*{Propositions 6.6
    and~7.2.4}).  Therefore, \(p (B\otimes\Comp) p\) has these four
  properties.  In particular, it is \(\K_1\)\nb-bijective.
\end{proof}


\section{The purely infinite case}
\label{sec:purely_inf}

At this point, our results combined with Kirchberg's Classification
Theorem imply a criterion when there is a \Cstar{}algebra
isomorphism \(\Cst(G) \cong B\), even without assuming~\(B\) to be a
graph \Cstar{}algebra:

\begin{theorem}
  \label{the:pi_graph_classify}
  Let~\(G\)
  be a countable graph for which~\(\Cst(G)\)
  is purely infinite.  Let~\(B\)
  be a separable, nuclear, strongly purely infinite \Cstar{}algebra
  such that all ideals of~\(B\)
  are in the bootstrap class and such that \(\Prim(B)\)
  has a basis consisting of quasi-compact open subsets.  Assume that
  there are an isomorphism \(\psi\colon\Idealsc(G)\to\Idealsc(B)\)
  of partially ordered sets and isomorphisms of diagrams
  \(\varphi_j\colon\K_j(\Cst(G),\Idealsc) \congto
  \psi^*\K_j(B,\Idealsc)\)
  for \(j=0,1\)
  for which the obstruction class vanishes.  Then \(\Cst(G)\)
  and~\(B\) are stably isomorphic, and vice versa.
\end{theorem}

The proof of the converse direction (``vice versa'') is easy.  The
proof of the interesting direction requires some preparatory
results.

\begin{lemma}
  \label{lem:from_compact_to_all_subsets}
  Let~\(X\) be a topological space.  Let
  \(\Openc(X)\subseteq\Open(X)\) be the set of compact-open subsets.
  Assume that~\(\Openc(X)\) is a basis.  Let \(\Fam(X)\) denote the
  set of subsets of \(\Openc(X)\) which are closed under finite
  unions and subsets.  Then the map
  \(U\mapsto F_U\defeq\{V\in\Openc(X)\mid V\subseteq U\}\) is an
  isomorphism of partially ordered sets \(\Open(X)\cong\Fam(X)\)
  with inverse \(S\mapsto\bigcup S\).
\end{lemma}

\begin{proof}
  The maps \(F\colon \Open(X)\to\Fam(X)\) and
  \({\bigcup}\colon \Fam(X) \to \Open(X)\) are well defined
  monotone maps.  One composite is the identity as
  \(U=\bigcup {}\setgiven{V\in\Openc(X)}{V\subseteq U}\) for all
  \(U\in\Open(X)\) by assumption on~\(X\); the other is the identity
  because
  \[
   S\subseteq\setgiven{V\in\Openc(X)}{V\subseteq\bigcup S}\subseteq S
  \]
  for all \(S\in\Fam(X)\),
  where the first inclusion is obvious and the second is seen as
  follows: if \(V\subseteq\bigcup S\)
  is compact, then \(V\subseteq\bigcup T\)
  for some finite subset \(T\subseteq S\)
  and then \(V\in S\) by assumption on~\(S\).
\end{proof}

\begin{proposition}
  \label{pro:from_compact_to_all_ideals}
  Let~\(B\)
  be a separable purely infinite \Cstar{}algebra such that the
  topology on~\(\Prim(B)\)
  has a basis consisting of compact-open subsets.  Let~\(\Fam(\Idealsc(B))\)
  denote the set of subsets of~\(\Idealsc(B)\)
  which are closed under finite unions and subsets.  Then
  \(I\mapsto F_I\defeq\{J\in\Idealsc(B)\mid J\subseteq I\}\)
  defines an isomorphism of partially ordered sets
  \(\Ideals(B)\cong\Fam(\Idealsc(B))\)
  with inverse \(S\mapsto\overline{\sum S}\).
\end{proposition}

\begin{proof}
  \cite{Pasnicu-Rordam:Purely_infinite_rr0}*{Proposition~2.7} says
  that~\(\Idealsc(B)\)
  is the set of all ideals \(I\idealin B\)
  for which~\(\widehat{I}\)
  is quasi-compact.  Now we can deduce the result from
  Lemma~\ref{lem:from_compact_to_all_subsets}, using the isomorphism
  \(\Open\bigl(\Prim(A)\bigr) \cong \Ideals(A)\),
  \(U \mapsto \bigcap\Prim(A)\setminus U\)
  (see \cite{Dixmier:Cstar-algebres}*{\S3.2}).
\end{proof}

\begin{proof}[Proof of Theorem~\textup{\ref{the:pi_graph_classify}}]
  Since \(\Cst(G)\) and~\(B\) are purely infinite,
  Corollary~\ref{cor:Hom_prop_purely_infinite} shows that an
  isomorphism
  \(\varphi_0\colon \K_0(\Cst(G),\Idealsc) \congto
  \psi^*\K_0(B,\Idealsc)\) is equivalent to an isomorphism
  \(\prop(\Cst(G)) \cong \prop(B)\).  Then the obstruction class in
  the statement is well defined.  If it vanishes,
  Proposition~\ref{pro:lift_to_K0_K1} gives a proper
  \(\Cst(G),B\)-correspondence \(\Hilm\) that induces the given
  isomorphisms~\(\varphi_0\) and~\(\varphi_1\).  The proper
  correspondence~\(\Hilm\) has a class in the equivariant KK-group
  \(\KK(\Prim \Cst(G); \Cst(G),B)\).  It induces an isomorphism on
  the \(\K\)\nb-theory of all ideals in
  \(\Idealsc(B) \cong \Idealsc(\Cst(G))\).  By the Universal
  Coefficient Theorem, it induces E\nb-theory equivalences on these
  ideals.

  Every ideal in \(\Cst(G)\)
  or in~\(B\)
  is generated by its projections.  This follows for~\(\Cst(G)\)
  because every ideal in a purely infinite graph \Cstar{}algebra is
  gauge-invariant and for~\(B\)
  by \cite{Pasnicu-Rordam:Purely_infinite_rr0}*{Proposition~2.7}.  It
  also follows from
  \cite{Pasnicu-Rordam:Purely_infinite_rr0}*{Proposition~2.7} that
  \(\Prim(\Cst(G))\)
  has a basis consisting of quasi-compact open subsets.  Hence
  Proposition~\ref{pro:from_compact_to_all_ideals} shows that~\(\psi\)
  extends canonically to a lattice isomorphism
  \(\Ideals(\Cst(G))\congto\Ideals(B)\).
  Thus \(\Prim(\Cst(G)) \cong \Prim(B)\).
  The quasi-compact open subsets form a basis for the topology on
  these primitive ideal spaces, and the correspondence~\(\Hilm\)
  induces E\nb-theory equivalences on those ideals that correspond to
  quasi-compact open subsets.  Thus
  \cite{Dadarlat-Meyer:E_over_space}*{Theorems 3.2 and~3.10} imply
  that~\(\Hilm\)
  induces a \(\Prim \Cst(G)\)-equivariant
  E\nb-theory equivalence.  Then the corresponding equivariant
  KK-class is also invertible by
  \cite{Gabe:Lifting_theorems_for_cp_maps}*{Theorem~6.2}.  Hence
  Kirchberg's Classification Theorem allows to lift it to a stable
  isomorphism between \(\Cst(G)\) and~\(B\).
\end{proof}

\begin{remark}
  \label{rem:stable_iso_purely_infinite}
  The constructed stable isomorphism
  \(\Cst(G)\otimes\Comp\cong B\otimes\Comp\)
  also gives a proper \(\Cst(G),B\)-correspondence.
  So it corresponds to a particular choice of the proper Hilbert
  \(B\)\nb-modules~\(\Hilm_v\)
  and the unitaries~\(U_v\)
  in our description of such correspondences.  Of course, it is very
  hard to find these unitaries directly.
\end{remark}

If the target $\Cst$\nb-algebra~$B$ is known to be a graph
$\Cst$\nb-algebra as well, then the proof of
Theorem~\ref{the:pi_graph_classify} simplifies considerably.  Then
Theorem~\ref{the:graph_homotopy_equivalence} and
Proposition~\ref{pro:from_compact_to_all_ideals} show that the
$\Cst$\nb-correspondence above is an ideal-preserving stable
homotopy equivalence.  This forces it to be an ideal-preserving
$\KK$-equivalence.

\section{Homotopy classes of correspondences}
\label{sec:homotopy_corres}

In this section, we first study when two proper
\(\Cst(G),B\)\nb-correspondences \(\Hilm\) and~\(\Hilm'\) are
homotopic.  Secondly, we use this to prove a criterion for two graph
\Cstar{}algebras to be homotopy equivalent.

\begin{lemma}
  Let \(\Hilm\) and~\(\Hilm'\) be homotopic.  Then they induce the
  same maps
  \begin{alignat*}{2}
    \varphi_0\colon \prop(\Cst(G)) &\to \prop(B),&\qquad
    \psi\colon \Idealsc^\infty(\Cst(G)) &\to \Idealsc^\infty(B),\\
    \varphi_1\colon \K_1(\Cst(G),\Idealsc) &\to
    \psi^*\K_1(B,\Idealsc).
  \end{alignat*}
  The entire homotopy~\((\Hilm_t)_{t\in[0,1]}\) preserves all ideals
  in \(\Idealsc^\infty(\Cst(G))\) in the sense that each~\(\Hilm_t\)
  induces the same map
  \(\Idealsc^\infty(\Cst(G)) \to \Idealsc^\infty(B)\).
\end{lemma}

\begin{proof}
  The maps \(\prop(\Cst(G)) \to \prop(B)\) induced by \(\Hilm\)
  and~\(\Hilm'\) are the same because homotopic projections are
  Murray--von Neumann equivalent.  Therefore, \(\Hilm\)
  and~\(\Hilm'\) --~and also~\(\Hilm_t\) for \(t\in [0,1]\)~--
  induce the same map
  \(\Idealsc^\infty(\Cst(G)) \to \Idealsc^\infty(B)\).
  Since~\(\K_1\) is defined using homotopy of unitaries, it follows
  that the diagram morphisms
  \(\K_1(\Cst(G),\Idealsc) \to \psi^*\K_1(B,\Idealsc)\) induced by
  \(\Hilm\) and~\(\Hilm'\) are equal as well.  Even more, they are
  constant along the homotopy~\((\Hilm_t)\).
\end{proof}

Let \(\Hilm\) and~\(\Hilm'\) be \(\Cst(G),B\)\nb-correspondences
that induce the same maps
\[
  \varphi_0\colon \prop(\Cst(G)) \to \prop(B),\qquad
  \varphi_1\colon \K_1(\Cst(G),\Idealsc) \to
  \psi^*\K_1(B,\Idealsc).
\]
The first condition implies that \(\Hilm\) and~\(\Hilm'\) also induce
the same map \(\psi\colon \Idealsc(\Cst(G)) \to \Idealsc(B)\), which
is needed to formulate the second condition.  Describe \(\Hilm\)
and~\(\Hilm'\) by families \(\bigl((\Hilm_v),(U_v)\bigr)\) and
\(\bigl((\Hilm_v'),(U'_v)\bigr)\), respectively, as in
Theorem~\ref{the:lift_prop_CB}.  Since \([\Hilm_v] = [\Hilm_v']\) in
\(\prop(B)\), the Hilbert \(B\)\nb-modules \(\Hilm_v\)
and~\(\Hilm'_v\) are isomorphic.  For simplicity, we
replace~\(\Hilm'\) by an isomorphic correspondence so that these
Hilbert modules become equal.  Then
Lemma~\ref{lem:difference_unitaries} gives unitaries
\(\Upsilon_v\in\Comp(\Hilm_v)\) for \(v\in V_\reg\) with
\(U_v' = \Upsilon_v U_v\).

Next, we describe an obstruction for \(\Hilm\) and~\(\Hilm'\) to be
homotopic.  This obstruction will turn out to be sharp under an
extra assumption on~\(B\).  Let \(v\in V_\reg\).  Let
\(\langle v \rangle \in \Idealsc(G)\) be the hereditary and
saturated subset generated by~\(v\).  Since this corresponds to the
ideal in~\(\Cst(G)\) generated by~\(p_v\), it follows that
\(\psi(\langle v\rangle)\idealin B\) is the ideal generated by the
matrix coefficients of vectors in~\(\Hilm_v\).  Thus~\(\Hilm_v\) is
a Morita--Rieffel equivalence between \(\Comp(\Hilm_v)\) and
\(\psi(\langle v\rangle)\).  As in the proof of
Proposition~\ref{pro:lift_to_K0_K1}, we let~\([\Upsilon_v]\) denote
the class of~\(\Upsilon_v\) in
\(\K_1(\psi(\langle v\rangle)) \cong \K_1(\Comp(\Hilm_v))\).  We
combine these classes for \(v\in V_\reg\) into a morphism of
diagrams
\([\Upsilon]\colon \Z[V_\reg,\Idealsc] \to \psi^* \K_1(B,\Idealsc)\)
as in~\eqref{eq:difference_class_as_diagram_morphism}.  Since
\(\Hilm\) and~\(\Hilm'\) induce the same map~\(\varphi_1\),
\eqref{eq:difference_class_boundary} implies
\([\Upsilon]\circ \kappa=0\).

We have described the group
\[
\Eext^1\bigl(G,\psi^*\K_1(B,\Idealsc)\bigr) \defeq
\Ext^1_{\Idealsc(G)}\bigl(\K_0(\Cst(G),\Idealsc),\psi^*\K_1(B,\Idealsc)\bigr)
\]
as the first cohomology group of the cochain complex
in~\eqref{eq:compute_Ext_Idealsc}.  The equation
\([\Upsilon]\circ \kappa=0\)
says that~\([\Upsilon]\)
is a cocycle and thus defines a class in
\(\Eext^1\bigl(G,\psi^*\K_1(B,\Idealsc)\bigr)\).

\begin{proposition}
  \label{pro:Ext1_to_homotopy}
  If \(\Hilm\) and~\(\Hilm^*\) are homotopic
  \(\Cst(G),B\)-correspondences, then~\([\Upsilon]\) becomes~\(0\)
  in \(\Eext^1\bigl(G,\psi^*\K_1(B,\Idealsc)\bigr)\).  In
  particular, \([\Upsilon]\) does not depend on the choices made in
  its construction.
\end{proposition}

\begin{proof}
  By the remarks after Lemma~\ref{lem:difference_unitaries}, \(\Hilm\)
  and~\(\Hilm'\)
  are homotopic if and only if we may find unitaries
  \(W_{v,0},W_{v,1}\in\U(\Hilm_v)\) for all \(v\in V_\reg\) so that
  \begin{align*}
    \Upsilon_{v,0}
    &\defeq W_{v,0} \circ U_v \circ
      \biggl( \bigoplus_{e\in E^v} W_{\s(e),0}^*\biggr)\circ U_v^*,
    \\
    \Upsilon_{v,1}
    &\defeq W_{v,1} \circ \Upsilon_v U_v \circ
      \biggl( \bigoplus_{e\in E^v} W_{\s(e),1}^*\biggr) \circ U_v^*
  \end{align*}
  are homotopic in \(\U(\Hilm_v)\).
  Equivalently, \(\Upsilon_v\) is homotopic to
  \begin{equation}
    \label{cocycle_unitaries}
    W_v\cdot U_v \biggl( \bigoplus_{e\in E^v} W_{\s(e)}^*\biggr) U_v^*,
  \end{equation}
  where \(W_v \defeq W_{v,1}^* W_{v,0}\).  This is true by
  assumption.  In the same way that the unitaries
  \(\Upsilon_v\in \U(\Hilm_v)\) for \(v\in V_\reg\) give a map
  \([\Upsilon]\colon \Z[V_\reg,\Idealsc] \to \psi^*
  \K_1(B,\Idealsc)\), the unitaries~\(W_v\) give a map
  \([W]\colon \Z[V,\Idealsc] \to \psi^* \K_1(B,\Idealsc)\).  The
  class of the unitary in~\eqref{cocycle_unitaries} in
  \(\K_1(\Comp(\Hilm_v)) \cong \K_1(\psi(W_v))\) is
  \([W_v] - \sum_{e\in E^v} {}[W_{\s(e)}]\).  Therefore,
  \([W]\circ (\id-M_G^{\Idealsc}) = [\Upsilon]\) as diagram
  morphisms \(\Z[V_\reg,\Idealsc] \to \psi^*\K_1(B,\Idealsc)\).
  This witnesses that~\([\Upsilon]\) represents zero in
  \(\Eext^1\bigl(G,\psi^*\K_1(B,\Idealsc)\bigr)\).
\end{proof}

\begin{theorem}
  \label{the:Ext1_to_homotopy}
  Describe a \(\Cst(G),B\)-correspondences~\(\Hilm\) by a family
  \(\bigl((\Hilm_v),(U_v)\bigr)\) as in
  Theorem~\textup{\ref{the:lift_prop_CB}}.  Let
  \[
    \varphi_0\colon \prop(\Cst(G)) \to \prop(B),\qquad
    \varphi_1\colon \K_1(\Cst(G),\Idealsc) \to
    \psi^*\K_1(B,\Idealsc)
  \]
  be the maps induced by~\(\Hilm\).  Assume that the unital
  \(\Cst\)\nb-algebras \(\Comp(\Hilm_v)\) for \(v\in V_\reg\) are
  \(\K_1\)\nb-bijective.  Then the construction in
  Proposition~\textup{\ref{pro:Ext1_to_homotopy}} gives a bijection
  between the group
  \(\Eext^1\bigl(G,\psi^*\K_1(B,\Idealsc)\bigr)\)
  and the set of homotopy classes of
  \(\Cst(G),B\)-correspondences that induce the same maps
  \(\varphi_0\) and~\(\varphi_1\).
\end{theorem}

\begin{proof}
  Let \(\Hilm'\) and~\(\Hilm''\) be
  \(\Cst(G),B\)-correspondences that induce the same maps
  \(\varphi_0\) and~\(\varphi_1\) as~\(\Hilm\).  Describe \(\Hilm'\)
  and~\(\Hilm''\) by families \(\bigl((\Hilm_v'),(U'_v)\bigr)\) and
  \(\bigl((\Hilm_v''),(U''_v)\bigr)\) and replace them by isomorphic
  correspondences with \(\Hilm''_v = \Hilm'_v=\Hilm_v\) for all
  \(v\in V\).  Then there are unitaries \(\Upsilon_v'\)
  and~\(\Upsilon_v''\) for \(v\in V_\reg\) with
  \(U_v' = \Upsilon'_v U_v\) and \(U_v'' = \Upsilon''_v U_v\).
  Assume that \([\Upsilon'] = [\Upsilon'']\) holds in
  \(\Eext^1\bigl(G,\psi^*\K_1(B,\Idealsc)\bigr)\).  Then there is a
  map \(\beta\colon \Z[V,\Idealsc] \to \psi^* \K_1(B,\Idealsc)\)
  with
  \(\beta\circ (\id-M_G^{\Idealsc}) = [\Upsilon''] - [\Upsilon'] =
  [(\Upsilon')^*\Upsilon'']\).  The proof of
  Proposition~\ref{pro:lift_to_K0_K1} shows that~\(\beta\) is
  specified by elements
  \(\beta_v \in \K_1(\psi(\langle v\rangle)) \cong
  \K_1(\Comp(\Hilm_v))\) for \(v\in V\).  Since~\(\Comp(\Hilm_v)\)
  is \(\K_1\)\nb-surjective, these \(\K_1\)\nb-classes are already
  represented by unitaries \(W_v \in \U(\Hilm_v)\).  The equation
  \(\beta\circ (\id-M_G^{\Idealsc}) = [(\Upsilon')^*\Upsilon'']\)
  says that the unitaries in~\eqref{cocycle_unitaries}
  and~\((\Upsilon_v')^*\Upsilon_v''\) represent the same class in
  \(\K_1(\psi(\langle v\rangle)) \cong \K_1(\Comp(\Hilm_v))\).
  Since~\(\Comp(\Hilm_v)\) is \(\K_1\)\nb-injective,
  \((\Upsilon_v')^*\Upsilon_v''\) and
  \(W_v\cdot U_v \biggl( \bigoplus_{e\in E^v} W_{\s(e)}^*\biggr)
  U_v^*\) are homotopic in~\(\Comp(\Hilm_v)\).  Similarly,
  \(\Upsilon'_v W_v\) is homotopic to~\(W_v \Upsilon'_v\).  So
  \(\Upsilon_v''\) is homotopic in \(\Comp(\Hilm_v)\) to
  \(W_v\cdot \Upsilon'_v U_v \biggl( \bigoplus_{e\in E^v}
  W_{\s(e)}^*\biggr) U_v^*\).  The latter family together
  with~\(\Hilm_v\) for \(v\in V\) describes a
  \(\Cst(G),B\)-correspondence that is isomorphic to~\(\Hilm'\), and
  the homotopy of unitaries gives a homotopy between the latter
  and~\(\Hilm''\).  So \(\Hilm'\) and~\(\Hilm''\) are homotopic.

  Next, take any class in
  \(\Eext^1\bigl(G,\psi^*\K_1(B,\Idealsc)\bigr)\).  This is
  represented by some diagram morphism
  \(\upsilon\colon \Z[V_\reg,\Idealsc] \to \psi^*\K_1(B,\Idealsc)\)
  with \(\upsilon \circ \kappa = 0\).  For \(v\in V_\reg\), the
  value of this map on \(\delta_v \in \Z[\langle v\rangle]\) is a
  class in
  \(\K_1(\psi(\langle v\rangle)) \cong \K_1(\Comp(\Hilm_v))\).  The
  latter is represented by a unitary
  \(\Upsilon_v\in \Comp(\Hilm_v)\) because \(\Comp(\Hilm_v)\) is
  \(\K_1\)\nb-bijective.  The family
  \(\bigl((\Hilm_v), (\Upsilon_v U_v)\bigr)\) describes a
  \(\Cst(G),B\)-correspondence~\(\Hilm'\) that induces the same
  map~\(\varphi_0\).  Since \(\upsilon \circ \kappa = 0\), it also
  induces the same diagram morphism
  \(\K_1(\Cst(G),\Idealsc) \to \psi^*\K_1(B,\Idealsc)\).  By
  construction, \(\Hilm'\) reproduces the class in
  \(\Eext^1\bigl(G,\psi^*\K_1(B,\Idealsc)\bigr)\) that we started
  with.
\end{proof}

Now let \(G\)
and~\(G'\)
be two graphs.
We denote the sets of vertices and edges of \(G,G'\)
by \(V,V'\)
and \(E,E'\),
respectively.  Let \(A\defeq \Cst(G)\) and
\(B\defeq \Cst(G')\).  Assume that there are an order isomorphism
\(\psi\colon \Idealsc(G) \cong \Idealsc(A) \congto \Idealsc(B) \cong
\Idealsc(G')\) and diagram isomorphisms
\begin{align*}
  \varphi_0^+\colon \K_0^+(A,\Idealsc) &\congto \psi^*\K_0^+(B,\Idealsc),\\
  \varphi_1\colon \K_1(A,\Idealsc) &\congto \psi^*\K_1(B,\Idealsc).
\end{align*}
Since graph \Cstar{}algebras
have stable weak cancellation by
Theorem~\ref{the:graph_K1-bijective} and monoid isomorphisms
preserve order units, \(\psi\)
and~\(\varphi_0^+\) come from a monoid isomorphism
\[
  \varphi_0\colon \prop(A) \congto \prop(B)
\]
by Theorem~\ref{the:prop_vs_K0plus}.  When is there a stable homotopy
equivalence between \(A\)
and~\(B\)
that induces the pair of maps \(\varphi_0\)
and~\(\varphi_1\)?
In order to lift this pair of maps to proper correspondences, we
need the obstruction classes of \((\varphi_0,\varphi_1)\)
and \((\varphi_0^{-1},\varphi_1^{-1})\)
to vanish.  We are going to prove that this condition is also
sufficient for a homotopy equivalence.  First, we prove that one of
these two obstruction classes vanishes if and only if the other does
so.

\begin{lemma}
  \label{lem:obstruction_inverse}
  The obstruction class of \((\varphi_0,\varphi_1)\)
  vanishes if and only if the obstruction class of
  \((\varphi_0^{-1},\varphi_1^{-1})\) vanishes.
\end{lemma}

\begin{proof}
  By symmetry, it suffices to prove one implication.  So assume that
  the obstruction class of \((\varphi_0,\varphi_1)\) vanishes.  Then
  Proposition~\ref{pro:lift_to_K0_K1} gives an
  \(A,B\)\nb-\alb{}correspondence~\(\Hilm_{A B}\) that lifts
  \((\varphi_0,\varphi_1)\) because graph \Cstar{}algebras are
  \(\K_1\)\nb-surjective by Theorem~\ref{the:graph_K1-bijective}.
  Theorem~\ref{the:lift_prop_CB} gives a
  \(B,A\)\nb-\alb{}correspondence~\(\Hilm_{B A}\) that
  lifts~\(\varphi_0^{-1}\), but we do not yet control the map
  on~\(\K_1\) induced by~\(\Hilm_{B A}\).  The composite
  \(B,B\)-correspondence \(\Hilm_{A B}\otimes_A \Hilm_{B A}\)
  induces the map \(\varphi_0\circ \varphi_0^{-1}=\id_{\prop(B)}\)
  on~\(\prop(B)\).  So we may use it to compute the obstruction
  class for the maps \(\id_{\prop(B)}\) and
  \(\id_{\K_1(B,\Idealsc)}\).  We may also use the identity
  correspondence on~\(B\) to compute this obstruction class.  Since
  the obstruction class is well defined by
  Proposition~\ref{pro:obstruction_class_well-defined}, both
  computations give the same result.  The first one gives
  \(\id_{\K_1(B,\Idealsc)} - \varphi_1 \circ \Hilm_{B A,*}\) because
  \(\Hilm_{A B,*} = \varphi_1\) on~\(\K_1\).  The second one gives
  just~\(0\).  Thus
  \(\id_{\K_1(B,\Idealsc)} - \varphi_1 \circ \Hilm_{B A,*}\)
  represents~\(0\) in \(\Eext^2\bigl(G',\K_1(B,\Idealsc)\bigr)\).
  This remains so when we compose with the
  isomorphism~\(\varphi_1^{-1}\).  That is,
  \(\varphi_1^{-1} - \Hilm_{B A,*} \colon \K_1(B,\Idealsc) \to
  (\psi^{-1})^*\K_1(A,\Idealsc)\) represents~\(0\) in
  \(\Eext^2\bigl(G',(\psi^{-1})^*\K_1(A,\Idealsc)\bigr)\).  At the
  same time, \(\varphi_1^{-1} - \Hilm_{B A,*}\) is exactly the
  obstruction class of~\((\varphi_0^{-1},\varphi_1^{-1})\).
\end{proof}

\begin{lemma}
  \label{lem:graph_back_forth_homotopy}
  A self-correspondence~$\Hilm_A$ of a graph $\Cst$\nb-algebra~$A$
  that induces the identity maps on \(\prop(A)\) and
  $\K_1(A,\Idealsc)$ is a stable homotopy equivalence.  In addition,
  all maps in the homotopies involved here induce the identity map
  on the lattices of gauge-invariant ideals.
\end{lemma}

\begin{proof}
  Since~$\Hilm_A$ induces the identity map on $\prop(A)$, we may
  assume without loss of generality that the underlying Hilbert
  modules~$\Hilm_v$ are~$p_v A$ for all vertices~$v$.
  Then~$\Hilm_A$ is described through unitaries
  $\Upsilon_v\in p_v A p_v = \Comp(\Hilm_v)$ for \(v\in V_\reg\).
  Namely, its representation in Theorem~\ref{the:lift_prop_CB}
  involves the family of Hilbert modules $(p_v A)_{v\in V}$ and
  unitaries $(\Upsilon_v u_v)_{v\in V_\reg}$ with the canonical
  unitaries~$u_v$.  Since $\Hilm_v = \Hilm_v' = p_v A$, the
  underlying right Hilbert module of~$\Hilm_A$ is isomorphic to
  $\bigoplus_{v} p_v A = A$.  Thus~$\Hilm_A$ is isomorphic to the
  correspondence associated to the \Star{}homomorphism
  $\sigma\colon A\to A$ defined by $\sigma(p_v) = p_v$ for $v\in V$
  and $\sigma(t_e) = \Upsilon_{r(e)} t_e$ for $e\in E$.

  Now let $\sigma'\colon A\to A$ be defined by the conditions
  $\sigma'(p_v) = p_v$ for all $v\in V$ and
  $\sigma'(t_e) = \Upsilon_{r(e)}^* t_e$ for all $e\in E$.  Then
  $\sigma'\circ \sigma(p_v) = p_v$ and
  $\sigma'\circ \sigma(t_e) = \Upsilon_{r(e)}^*
  \sigma(\Upsilon_{r(e)})t_e$.  Since~$\sigma$ induces the identity
  map on $\K_1(A,\Idealsc)$, the unitaries
  $\Upsilon_{r(e)}^* \sigma(\Upsilon_{r(e)}) \in p_v A p_v$
  represent zero in $\K_1(p_v A p_v)$.  Since $p_v \Cst(A) p_v$ is
  \(\K_1\)\nb-injective for any vertex~$v$ of the graph by
  Theorem~\ref{the:graph_K1-bijective}, the unitary
  $\Upsilon_{r(e)}^* \sigma(\Upsilon_{r(e)})$ in $p_v A p_v$ is
  homotopic to the unit element~$p_v$.  This gives us a homotopy
  between $\sigma'\circ \sigma$ and the identity map (compare
  Theorem~\ref{the:Ext1_to_homotopy}).  Since $\sigma$ and
  $\sigma'\circ\sigma$ induce the identity map on the invariants, it
  follows that~$\sigma'$ does so as well.  Then we may switch the
  roles of $\sigma$ and~$\sigma'$ to show that $\sigma\circ \sigma'$
  is homotopic to the identity map as well.  Then~$\sigma$ is a
  homotopy equivalence.  Since all gauge-invariant ideals are
  generated by the projections that they contain, and since all maps
  in the homotopies induce the same maps on \(\prop(A)\), it follows
  that all maps in the homotopies induce the identity map on the
  lattices of gauge-invariant ideals.
\end{proof}

\begin{theorem}
  \label{the:graph_homotopy_equivalence}
  Let $G$ and~$G'$ be countable graphs.  Define
  \(A\defeq \Cst(G)\), \(B\defeq \Cst(G')\).  Let
  \(\psi\colon \Idealsc(G) \cong \Idealsc(A) \congto \Idealsc(B)
  \cong \Idealsc(G')\) be an order isomorphism and let
  \begin{align*}
    \varphi_0^+\colon \K_0^+(A,\Idealsc) &\congto \psi^*\K_0^+(B,\Idealsc),\\
    \varphi_1\colon \K_1(A,\Idealsc) &\congto \psi^*\K_1(B,\Idealsc)
  \end{align*}
  be isomorphism of diagrams.  These isomorphisms induce a monoid
  isomorphism \(\varphi_0\colon \prop(A) \congto \prop(B)\) as
  above.  Assume that the obstruction class of
  $(\varphi_0,\varphi_1)$ vanishes.  Then there is a stable homotopy
  equivalence between $\Cst(G)$ and $\Cst(G')$ that induces the maps
  $\varphi_0^+$ and~$\varphi_1$.  The homotopies in this equivalence
  automatically preserve gauge-invariant ideals, that is, the
  induced maps on \(\Ideals^\T(\Cst(G))\) and
  \(\Ideals^\T(\Cst(G'))\) are constant along the homotopies.
\end{theorem}

\begin{proof}
  Since the obstruction class of $(\varphi_0,\varphi_1)$ vanishes,
  the same holds for the inverse maps
  \((\varphi_0^{-1},\varphi_1^{-1})\) by
  Lemma~\ref{lem:obstruction_inverse}.  Since graph \Cstar{}algebras
  are \(\K_1\)\nb-surjective by
  Theorem~\ref{the:graph_K1-bijective},
  Proposition~\ref{pro:lift_to_K0_K1} gives an
  \(A,B\)-correspondence \(\Hilm_{A B}\) and a
  \(B,A\)-correspondence \(\Hilm_{B A}\) that induce the maps
  \((\varphi_0,\varphi_1)\) and \((\varphi_0^{-1},\varphi_1^{-1})\)
  on the invariants.  Let
  \(\Hilm_A \defeq \Hilm_{A B} \otimes_B \Hilm_{B A}\) and
  \(\Hilm_B \defeq \Hilm_{B A} \otimes_A \Hilm_{A B}\) be the
  composite \(A,A\)- and \(B,B\)-correspondences.  Both induce the
  identity maps on \(\prop\) and on the \(\K_1\)\nb-diagram.  Thus
  they are stable homotopy equivalences by
  Lemma~\ref{lem:graph_back_forth_homotopy}.  Thus $\Hilm_{A B}$ and
  \(\Hilm_{B A}\) are both left and right invertible up to homotopy.
  This forces them to be stable homotopy equivalences.  In addition,
  the lemma shows that the homotopies implicit in this all induce
  the identity map on the lattice of gauge-invariant ideals.
\end{proof}

\section{The circle components}
\label{sec:circles}

The homotopy equivalence in
Theorem~\ref{the:graph_homotopy_equivalence} has the extra property
that it preserves the gauge-invariant ideals.  A graph
\Cstar{}algebra may have gauge-invariant subquotients that are
Morita equivalent to \(\Cont(\T)\).  In this case, there are many
ideals that are not gauge-invariant.  In this section, we show that
we may arrange the homotopy equivalence to also preserve these
ideals.  As before, we let~\(G\) be a row-finite graph.  We are
going to use some results of Hong and
Szyma\'nski~\cite{Hong-Szymanski:Primitive_graph}.  Therefore, we
assume~\(G\) to be countable from now on.  This ensures that
ideals in \(\Cst(G)\) are prime if and only if they are primitive,
and this may go wrong in general (see
\cites{Abrams-Tomforde:prime_primitive,
  Abrams-Tomforde:Prime_primitive_graph}).  The following notation
is introduced in~\cite{Hong-Szymanski:Primitive_graph} to describe
the primitive ideal space of~\(\Cst(G)\).  Beware
that~\cite{Hong-Szymanski:Primitive_graph} uses the opposite
convention regarding the source and range maps on the graph.

\begin{definition}
  Write \(w\le v\) for \(w,v \in V\) if there is a path
  in~\(G\) from~\(w\) to~\(v\), and \(w\sim v\) if \(w\le v\) and
  \(v\le w\).
\end{definition}

The relation~\(\le\) is a preorder.  So~\(\sim\) is an equivalence
relation.

\begin{definition}
  A subset \(M\subseteq V\) is called a \emph{maximal tail} if
  \begin{enumerate}[label=\textup{(MT\arabic*)}]
  \item \label{en:maxtail1}%
    if \(e\in E\) and \(s(e)\in M\), then \(r(e)\in M\);
  \item \label{en:maxtail2}%
    if \(v\in M \cap V_\reg\), then there is \(e\in E\) with
    \(r(e)=v\) and \(s(e) \in M\);
  \item \label{en:maxtail3}%
    for all \(v,w\in M\), there is \(y\in M\) with \(v\ge y\) and
    \(w\ge y\).
  \end{enumerate}
  If~\(M\) is a maximal tail, let \(\Omega(M) = V\setminus M\).  Let
  \(\mathcal{M}(G)\) be the set of all maximal tails in~\(G\).  Let
  \(\mathcal{M}_\gamma(G) \subseteq \mathcal{M}(G)\) be the subset
  of all \(M\in \mathcal{M}(G)\) with the extra property that each
  cycle in~\(G|_M\) has an entry in~\(G|_M\).  Let
  \(\mathcal{M}_\tau(G) \defeq \mathcal{M}(G)\setminus
  \mathcal{M}_\gamma(G)\).
\end{definition}

Condition~\ref{en:maxtail1} holds if and only if \(v\ge w\) and
\(w\in M\) implies \(v\in M\).  Conditions \ref{en:maxtail1}
and~\ref{en:maxtail2} say exactly that \(\Omega(M)\) is hereditary
and saturated.  It is shown in~\cite{Hong-Szymanski:Primitive_graph}
that the gauge-invariant ideal corresponding to a hereditary
saturated subset~\(W \in \Ideals(G)\) is prime in the complete
lattice \(\Ideals^\T(\Cst(G))\) of gauge-invariant ideals if and
only if~\(V\setminus W\) satisfies~\ref{en:maxtail3} as well as
\ref{en:maxtail1} and~\ref{en:maxtail2}.  Equivalently, the map
\(M\mapsto \Omega(M)\) is a bijection between \(\mathcal{M}(G)\) and
the set of prime ideals in the lattice \(\Ideals(G)\).

\begin{proposition}[\cite{Hong-Szymanski:Primitive_graph}]
  \label{pro:Hong-Szymanski}
  Let \(M\in \mathcal{M}_\tau(G)\).  Then~\(M\) contains a unique
  simple cycle \(e_1 e_2 \dotsm e_k\), and this cycle has no entry
  in~\(M\).  Let \(v\in M\)
  be a vertex on this cycle.  Then
  \(\Cst(G)_{\Omega(M)} \idealin \Cst(G)\) is a prime ideal in
  \(\Cst(G)\) and the subquotient
  \[
    \Cst(G)|_M \defeq
    \Cst(G)_{\langle\Omega(M) \cup \{v\}\rangle} /
    \Cst(G)_{\Omega(M)}
  \]
  is gauge-simple and Morita--Rieffel equivalent to \(\Cont(\T)\).
  More precisely, the image of the corner \(p_v \Cst(G) p_v\) in
  \(\Cst(G) / \Cst(G)_{\Omega(M)}\) is isomorphic to \(\Cont(\T)\),
  generated by the image of \(t_{e_1} t_{e_2} \dotsm t_{e_k}\),
  which is unitary.  We call \(\Cst(G)|_M\) \emph{the gauge-simple
    subquotient of \(\Cst(G)\) associated to
    \(M\in\mathcal{M}_\tau(G)\).}
\end{proposition}

The corner \(p_v (\Cst(G)/ \Cst(G)_{\Omega(M)}) p_v\) is
Morita--Rieffel equivalent to the ideal in
\(\Cst(G)/ \Cst(G)_{\Omega(M)}\) generated by the image of~\(p_v\).
This ideal is gauge-invariant and corresponds to the smallest
hereditary and saturated subset of~\(V\) containing \(\Omega(M)\)
and~\(v\).  Then it must be \(\Cst(G)|_M\).

Now we return to the situation of
Theorem~\ref{the:graph_homotopy_equivalence}.  That is, $G$
and~$G'$ are countable graphs, \(A=\Cst(G)\), \(B=\Cst(G')\), and
there are an order isomorphism
\(\psi\colon \Idealsc(G) \cong \Idealsc(A) \congto \Idealsc(B) \cong
\Idealsc(G')\) and diagram isomorphisms
\begin{align*}
  \varphi_0^+\colon \K_0^+(A,\Idealsc)
  &\congto \psi^*\K_0^+(B,\Idealsc),\\
  \varphi_1\colon \K_1(A,\Idealsc)
  &\congto \psi^*\K_1(B,\Idealsc).
\end{align*}
These isomorphisms induce a monoid isomorphism
\(\varphi_0\colon \prop(A) \congto \prop(B)\).  We assume that the
obstruction class of $(\varphi_0,\varphi_1)$ vanishes.

\begin{lemma}
  \label{lem:iso_on_circle_primes}
  The order isomorphism
  \(\psi\colon \Idealsc(G) \congto \Idealsc(G')\) extends uniquely
  to an order isomorphism
  \(\bar\psi\colon \Ideals(G) = \Idealsc^\infty(G) \congto
  \Idealsc^\infty(G') = \Ideals(G')\).  The extension maps the
  subsets \(\mathcal{M}_\gamma(G)\) and \(\mathcal{M}_\tau(G)\)
  bijectively onto \(\mathcal{M}_\gamma(G')\) and
  \(\mathcal{M}_\tau(G')\), respectively.
\end{lemma}

\begin{proof}
  As in the proof of
  Proposition~\ref{pro:from_compact_to_all_ideals}, the
  extension~\(\bar\psi\) exists by
  Lemma~\ref{lem:from_compact_to_all_subsets}.  The order
  isomorphism~\(\bar\psi\) restricts to a bijection between the
  prime ideals.  Thus it induces a bijection
  \(\psi\colon \mathcal{M}(G)\congto\mathcal{M}(G')\).  We must show
  that \(M\in\mathcal{M}(G)\) belongs to \(\mathcal{M}_\tau(G)\) if
  and only if \(\psi(M)\in\mathcal{M}(G')\) belongs to
  \(\mathcal{M}_\tau(G')\).  To do this, we
  characterise through the diagrams \(\K_0^+(\Cst(G),\Idealsc)\) and
  \(\K_1(\Cst(G),\Idealsc)\) when \(M\in\mathcal{M}(G)\) belongs to
  \(\mathcal{M}_\tau(G)\).  First, Remark~\ref{rem:Idealsc_infty}
  shows that \(\varphi_0^+\) and~\(\varphi_1\) extend to
  isomorphisms of diagrams
  \(\K_0^+(\Cst(G),\Idealsc^\infty)\congto\K_0^+(\Cst(G'),\Idealsc^\infty)\)
  and
  \(\K_1(\Cst(G),\Idealsc^\infty)\congto
  \K_1(\Cst(G'),\Idealsc^\infty)\).  Theorem~\ref{the:graph_Idealsc}
  identifies \(\Idealsc^\infty(\Cst(G))\) with~\(\Ideals(G)\), the
  lattice of hereditary saturated subsets.

  Let \(M\in\mathcal{M}(G)\).  Then
  \(H\defeq \Omega(M) = V\setminus M \in\Ideals(G)\).  If
  \(M\in\mathcal{M}_\tau(G)\), then we define
  \(H_2 \defeq \langle\Omega(M) \cup \{v\}\rangle \in\Ideals(G)\) as
  in Proposition~\ref{pro:Hong-Szymanski}.  Then
  \(H\subsetneq H_2\).  If \(K\in\Ideals(G)\) also satisfies
  \(H\subsetneq K\), then \(K\cap H_2 = H\) or \(K\cap H_2 = H_2\)
  because \(K\cap H_2 \in \Ideals(G)\) and the subquotient
  \(\Cst(G)|_M \defeq \Cst(G)_{H_2} / \Cst(G)_H\) has no
  gauge-invariant ideals.  The case \(K\cap H_2 = H\) is impossible
  because~\(H\) is prime.  It follows that \(H_2\in \Ideals(G)\) is
  the unique minimal element of the set
  \(\setgiven{K\in\Ideals(G)}{H\subsetneq K}\).  Therefore, the
  existence of such a unique minimal element~\(H_2\) is necessary
  for \(M\in\mathcal{M}_\tau(G)\).  This property is preserved by
  order isomorphisms.  We assume from now on that such an~\(H_2\)
  exists and, in this generality, we define
  \[
    \Cst(G)|_M \defeq \Cst(G)_{H_2} / \Cst(G)_H.
  \]

  A quotient of a graph \Cstar{}algebra by a gauge-invariant ideal
  is again a graph \Cstar{}algebra (see
  \cite{Bates-Pask-Raeburn-Szymanski:Row_finite}*{Thereom~4.1}), and
  the action induced by the gauge action is the usual gauge action.
  Any gauge-invariant ideal in a graph \Cstar{}algebra is again
  isomorphic to a graph \Cstar{}algebra by a gauge-equivariant
  isomorphism (see~\cite{Vas:Graded_ideal_Leavitt}).
  Therefore, \(\Cst(G)|_M\) is
  \(\T\)\nb-equivariantly isomorphic to the graph
  \Cstar{}algebra of some graph~\(G_M\).  This subquotient
  of~\(\Cst(G)\) has no gauge-invariant ideals because~\(H_2\) is
  the unique minimal element of
  \(\setgiven{K\in\Ideals(G)}{H\subsetneq K}\).  Equivalently, the
  graph~\(G_M\) is cofinal.  Now there are three possible cases:
  \begin{enumerate}
  \item \(G_M\) has no cycles; then \(\Cst(G_M)\) is a simple
    AF-algebra;
  \item \(G_M\) has more than one simple cycle; then \(\Cst(G_M)\) is
    simple purely infinite;
  \item \(G_M\) contains exactly one simple cycle.
  \end{enumerate}
  In the third case, the unique simple cycle cannot have an entry
  because then \(\Cst(G_M)\) would be simple by
  \cite{Bates-Hong-Raeburn-Szymanski:Ideal_structure}*{Proposition~5.1},
  and then we must be in one of the first two cases, by the well
  known dichotomy for simple graph \Cstar{}algebras.  So
  \(M\in\mathcal{M}_\tau(G)\) if and only if
  \(\Cst(G_M)\) is Morita--Rieffel equivalent to \(\Cont(\T)\).

  The graph~\(G_M\) is described explicitly in
  \cite{Bates-Pask-Raeburn-Szymanski:Row_finite}
  and~\cite{Vas:Graded_ideal_Leavitt}.  Its vertex set is obtained
  from \(M = H_2 \setminus H\) by adding certain vertices.  An
  inspection shows that the projections~\(p_v\) for the added
  vertices are all Murray--von Neumann equivalent to projections
  from vertices in~\(M\).
  Therefore, Theorem~\ref{the:prop_graph} implies that any
  projection in \(\Cst(G_M)\) lifts to a projection in
  \(\Cst(G)_{H_2}\).  That is, the map
  \(\K_0^+(\Cst(G)_{H_2}) \to \K_0^+(\Cst(G_M))\) is surjective.
  This remains so on the Grothendieck groups.  Therefore, the
  \(\K\)\nb-theory long exact sequence for the \Cstar{}algebra
  extension \(\Cst(G)_H \into \Cst(G)_{H_2} \prto \Cst(G)|_M\) looks
  as follows:
  \begin{equation}
    \label{eq:exact_sequence_H_H2}
    \begin{tikzcd}
      \K_0(\Cst(G)_H) \ar[r, "i_0"] &
      \K_0(\Cst(G)_{H_2}) \ar[r, ->>] &
      \K_0(\Cst(G)|_M) \ar[d, "0"] \\
      \K_1(\Cst(G)|_M) \ar[u] &
      \K_1(\Cst(G)_{H_2}) \ar[l, "i_1"] &
      \K_1(\Cst(G)_H). \ar[l, >->]
    \end{tikzcd}
  \end{equation}
  If \(M\in\mathcal{M}_\tau(G)\), then \(\Cst(G)|_M\) is
  Morita--Rieffel equivalent to \(\Cont(\T)\) by
  Proposition~\ref{pro:Hong-Szymanski}
  and therefore
  \(\K_0^+(\Cst(G)|_M) \cong \N\) and \(\K_1(\Cst(G)|_M) \cong \Z\).
  In contrast, if \(M\not\in\mathcal{M}_\tau(G)\), then either
  \(\Cst(G)|_M\) is purely infinite, so that \(\K_0^+(\Cst(G)|_M)\) is a
  group, or \(\Cst(G)|_M\) is AF, so that \(\K_1(\Cst(G)|_M) \cong 0\).
  Now we check that these three cases can be distinguished looking
  only at the maps
  \[
    \K_0^+(\Cst(G)_H) \to \K_0^+(\Cst(G)_{H_2}),\qquad
    \K_i(\Cst(G)_H) \to \K_i(\Cst(G)_{H_2})
  \]
  for \(i=0,1\) induced by the inclusion \(H\subsetneq H_2\).
  Indeed, if \(M\in\mathcal{M}_\tau(G)\), then the quotient of
  \(\K_0^+(\Cst(G)_{H_2})\) by the image of \(\K_0^+(\Cst(G)_H)\)
  cannot be a group because it maps to the monoid
  \(\K_0^+(\Cst(G)|_M)\cong \N\), which is not a group,
  and it cannot happen that the map
  \(\K_0(\Cst(G)_H) \to \K_0(\Cst(G)_{H_2})\) is injective and
  \(\K_1(\Cst(G)_H) \to \K_1(\Cst(G)_{H_2})\) is surjective, because
  that would force \(\K_1(\Cst(G)|_M)\) to vanish.  If
  \(M\not\in\mathcal{M}_\tau(G)\) and \(\Cst(G)|_M\) is
  purely infinite simple, then the quotient of
  \(\K_0^+(\Cst(G)_{H_2})\) by the image of \(\K_0^+(\Cst(G)_H)\) is
  a group.  And if \(M\not\in\mathcal{M}_\tau(G)\) and
  \(\Cst(G)|_M\) is AF, then the map
  \(\K_0(\Cst(G)_H) \to \K_0(\Cst(G)_{H_2})\) is injective and
  \(\K_1(\Cst(G)_H) \to \K_1(\Cst(G)_{H_2})\) is surjective.  Thus
  we can read off from \(\Ideals(G)\) and the diagrams
  \(\K_0^+(\Cst(G),\Ideals)\) and \(\K_1(\Cst(G),\Ideals)\) whether
  or not \(M\not\in\mathcal{M}_\tau(G)\).  Since these invariants do
  not distinguish our two graphs, it follows that the bijection
  \(\psi\colon \mathcal{M}(G)\congto\mathcal{M}(G')\) maps
  \(\mathcal{M}_\tau(G)\) onto \(\mathcal{M}_\tau(G')\).
\end{proof}

\begin{theorem}
  \label{the:classify_graph_with_circles}
  Let $G$ and~$G'$ be countable graphs.  Define \(A\defeq \Cst(G)\),
  \(B\defeq \Cst(G')\).  Let
  \begin{align*}
    \psi\colon \Idealsc(G) \cong \Idealsc(A)
    &\congto \Idealsc(B) \cong \Idealsc(G')
    \\\shortintertext{be an order isomorphism and let}
    \varphi_0^+\colon \K_0^+(A,\Idealsc)
    &\congto \psi^*\K_0^+(B,\Idealsc),\\
    \varphi_1\colon \K_1(A,\Idealsc)
    &\congto \psi^*\K_1(B,\Idealsc)
  \end{align*}
  be diagram isomorphisms.  These isomorphisms induce a monoid
  isomorphism \(\varphi_0\colon \prop(A) \congto \prop(B)\).  Assume
  that the obstruction class of $(\varphi_0,\varphi_1)$ vanishes.
  Then there is an \(A\)\nb-\(B\)-correspondence~\(\Hilm_{A B}\)
  that induces the isomorphisms $\varphi_0^+$ and~$\varphi_1$ and
  that induces a Morita--Rieffel equivalence between the
  gauge-simple subquotients of \(A\) and~\(B\) associated to \(M\)
  and \(\psi(M)\) for all \(M\in \mathcal{M}_\tau(G)\).  In
  addition, there is a \(B\)\nb-\(A\)-correspondence~\(\Hilm_{B A}\)
  that is inverse to~\(\Hilm_{A B}\) up to homotopy, and such that
  all the correspondences in the homotopies between the composite
  correspondences \(\Hilm_{A B} \otimes_B \Hilm_{B A}\) and
  \(\Hilm_{B A} \otimes_A \Hilm_{A B}\) and the identity
  correspondences induce the identity correspondences on the
  gauge-simple subquotients associated to all
  \(M\in \mathcal{M}_\tau(G)\), and similarly for
  \(\Hilm_{B A} \otimes_A \Hilm_{A B}\).
\end{theorem}

\begin{proof}
  First, we fix \(M\in \mathcal{M}_\tau(G)\) and write it as
  \(M = H_2 \setminus H\) with hereditary and saturated subsets
  \(H\) and~\(H_2\) as above.  Let \(H_2' \defeq \psi(H_2)\) and
  \(H' \defeq \psi(H)\) in \(\Ideals(G')\), and let
  \(M'\defeq H_2' \setminus H'\).  Then
  \(M'\in \mathcal{M}_\tau(G')\) by
  Lemma~\ref{lem:iso_on_circle_primes}.

  Theorem~\ref{the:graph_homotopy_equivalence} shows that there is a
  \(\Cst(G)\)-\(\Cst(G')\)-correspondence~\(\Hilm_{A B}\) that
  induces the maps \(\psi\), \(\varphi_0^+\) and~\(\varphi_1\) on
  our invariants.  In particular, this correspondence induces
  correspondences between the gauge-invariant ideals \(\Cst(G)_H\)
  and \(\Cst(G')_{\psi(H)}\) for all \(H\in\Idealsc(G)\); this
  extends automatically to \(H\in \Ideals(G)\).  And then it also
  induces a proper correspondence~\(\Hilm_{A B,M}\) between the
  subquotients \(\Cst(G)|_M\) and \(\Cst(G')|_{M'}\).  Since
  our correspondence is a homotopy equivalence that preserves the
  gauge-invariant ideals, it induces an isomorphism on
  \(\K_*(\Cst(G)_H)\) and \(\K_*(\Cst(G)_{H_2})\)
  in~\eqref{eq:exact_sequence_H_H2}.  By the Five Lemma, it induces
  an isomorphism
  \(\K_*(\Cst(G)|_M) \congto \K_*(\Cst(G')|_{M'})\) as well.
  In addition, this isomorphism on \(\K_0(\Cst(G)|_M) \cong \Z\)
  maps the positive cone~\(\N\) into itself.  Thus it is the
  identity map.
  
  Let \(v\in M\) and \(v' \in M'\) be vertices on the unique cycles
  in \(M\) and~\(M'\), respectively.  The image of~\(p_v\) generates
  \(\K_0^+(\Cst(G)|_M)\), and any other idempotent that generates
  \(\K_0^+(\Cst(G)|_M)\) is Murray--von Neumann equivalent to it
  because~\(\Cst(G)|_M\) is Morita equivalent to \(\Cont(\T)\) by
  Proposition~\ref{pro:Hong-Szymanski} and
  \(\prop(\Cont(\T)) \cong \N\).  Similarly, any generator of
  \(\K_0^+(\Cst(G')|_{M'})\) is Murray--von Neumann equivalent
  to~\(p_{v'}\) in \(\Cst(G')_{M'}\), and the corner in
  \(\Cst(G')_{M'}\) that it cuts out is isomorphic to~\(\Cont(\T)\)
  as well.  Since the projective module~\(p_v \Hilm_{A B,M}\)
  over~\(\Cst(G')_{M'}\) is a generator of
  \(\prop(\Cst(G')|_{M'})\), it is isomorphic to
  \(\Cst(G')|_{M'} p_{v'}\).  We fix such an isomorphism
  \(W\colon p_v \Hilm_{A B,M} \congto \Cst(G')|_{M'} p_{v'}\).  Let
  \(z \in \Cont(\T) \cong p_v \Cst(G)|_M p_v \cong \Cont(\T)\) be
  the canonical unitary generator, and also write~\(z\) for the
  operator of left multiplication by~\(z\) on~\(p_v \Hilm_{A B,M}\).
  Then \(W z W^*\) is a unitary in
  \(\Comp(\Cst(G')|_{M'} p_{v'}) \cong p_{v'}\Cst(G')|_{M'} p_{v'}
  \cong \Cont(\T)\).  Since~\(\Hilm_{A B,M}\) induces an isomorphism
  on~\(\K_1\), the image of~\(W z W^*\) in
  \(\K_1(\Cont(\T)) \cong \Z\) is a generator.  Since \(\Cont(\T)\)
  is \(\K_1\)\nb-bijective, say, by
  Theorem~\ref{the:graph_K1-bijective}, \(W z W^*\) is homotopic to
  either~\(z\) or~\(z^{-1}\).  Then we may write
  \(W z W^* = \bar\zeta \cdot z^{\pm1}\) with a
  unitary~\(\bar\zeta\) in \(p_{v'} \Cst(G')|_{M'} p_{v'}\) that is
  homotopic to the unit.  Then the unitary
  \(W^* \bar\zeta^* W \in \Comp(p_v \Hilm_{A B,M})\) is homotopic
  to~\(1\) as well.  Therefore, we may lift \(W^* \bar\zeta^* W\) to
  a unitary \(\Upsilon_v \in \Comp(p_v \Hilm_{A B})|_{H_2}\)
  homotopic to~\(1\) along the quotient map to
  \(\Comp(p_v \Hilm_{A B,M})\).  When we multiply the
  unitary~\(U_v\) in our description of~\(\Hilm_{A B}\)
  by~\(\Upsilon_v\), we do not change the homotopy class
  of~\(\Hilm_{A B}\).  Thus the induced maps on \(\K_0^+\)
  and~\(\K_1\) also remain the same.  By
  Proposition~\ref{pro:Hong-Szymanski}, the element
  \(z\in \Cont(\T) \subseteq p_v \Cst(G|_M) p_v\) is the product of
  the edge generators in the cycle in~\(M\) based at~\(v\).  So it
  contains a matrix coefficient of~\(U_v\) exactly once.  Therefore,
  after multiplying~\(U_v\) by~\(\Upsilon_v\), we have arranged that
  the induced map
  \begin{equation}
    \label{eq:star_iso_induced_by_HAB}
    \Cont(\T)
    \cong p_v \Cst(G)|_M p_v
    \to \Comp(\Hilm_{A B, M})
    \cong p_{v'} \Cst(G')|_{\psi(M)} p_{v'}
    \cong \Cont(\T)
  \end{equation}
  is the isomorphism induced by \(z\mapsto z^{\pm1}\).
  Thus~\(\Hilm_{A B,M}\) has become a Morita equivalence bimodule.

  The only ingrediet of~\(\Hilm_{B A}\) that we have changed to
  achieve this is the unitary~\(U_v\) for one vertex \(v\in M\).
  These changes for different \(M\in\mathcal{M}_\tau(G)\) do not
  interfere with each other.  Therefore, we may arrange
  for~\(\Hilm_{A B,M}\) to become a Morita equivalence bimodule for
  all \(M\in\mathcal{M}_\tau(G)\) simultaneously.
  
  Lemma~\ref{lem:graph_back_forth_homotopy} gives a
  \(B\)\nb-\(A\)-correspondence~\(\Hilm_{B A}\) that is inverse
  to~\(\Hilm_{A B}\) up to homotopy.  Even more, the homotopies that
  are implicit here automatically preserve all gauge-invariant
  ideals because equivalent projections generate the same ideal.
  Therefore, for each \(M\in\mathcal{M}_\tau(G)\), there are induced
  correspondences \(\Hilm_{A B,M}\) and~\(\Hilm_{B A,M}\) between
  \(\Cst(G)|_M\) and \(\Cst(G')|_{\psi(M)}\) that are inverse to
  each other up to homotopy.
  
  In the same way as above, we may replace~\(\Hilm_{B A}\) by a
  homotopic proper correspondence that induces Morita equivalence
  bimodules on the subquotients for all
  \(M'\in\mathcal{M}_\tau(G')\).  Even more, we arranged both
  \(\Hilm_{A B}\) and \(\Hilm_{B A}\) to map the preferred
  generator~\(z\) of \(\Cont(\T)\) to \(z\) or~\(z^{-1}\).  Hence so
  do the composites \(\Hilm_{A B} \otimes_B \Hilm_{B A}\)
  and \(\Hilm_{B A} \otimes_A \Hilm_{A B}\).  Since these
  composites are homotopic to the identity correspondence, the map
  \(z\mapsto z^{-1}\) cannot occur here.  Therefore, we know that
  the composite correspondences
  \(\Hilm_{A B} \otimes_B \Hilm_{B A}\) and
  \(\Hilm_{B A} \otimes_A \Hilm_{A B}\) both induce the
  identity bimodules on the subquotients for elements of
  \(\mathcal{M}_\tau(G)\).  For each \(M\in \mathcal{M}_\tau(G)\),
  the homotopy between
  \(\Hilm_{A B} \otimes_B \Hilm_{B A}\) and the identity
  correspondence on~\(A\) induces a homotopy of \Star{}homomorphisms
  \(\Cont(\T) \to \Cont(\T)\) of the form
  \(z\mapsto \bar\zeta_t \cdot z\) for some continuous path of
  unitaries~\(\bar\zeta_t\), where we already arranged that
  \(\bar\zeta_t=1\) for \(t=0,1\).  We may lift these unitaries as
  above and correct the \(\Cst(G)\)-\(\Cst(G)\)-correspondences
  \((\Hilm_t)_{t\in[0,1]}\) in the homotopy, so that \(\Hilm_0\)
  and~\(\Hilm_1\) remain the same and each~\(\Hilm_t\) induces the
  identity map on the corner \(\Cont(\T) \cong p_v \Cst(G)|_M p_v\)
  for all \(M\in\mathcal{M}_\tau(G)\) and some \(v\in M\) in the
  unique cycle in~\(M\).  This yields a new homotopy between
  \(\Hilm_{A B} \otimes_B \Hilm_{B A}\) and the identity
  correspondence, with the extra property that each correspondence
  in the homotopy induces the identity Hilbert bimodule on
  \(\Cst(G)|_M\) for all \(M\in\mathcal{M}_\tau(G)\).  The same
  change of homotopy is possible for the other composite
  \(\Hilm_{B A} \otimes_A \Hilm_{A B}\).
\end{proof}

\end{document}